\documentclass[12pt,amstex,amssymb]{amsart}
\usepackage{graphicx}
\usepackage{tikz}
\usepackage{tikz-cd}
\usepackage{framed}
\usepackage{amsmath}
\usepackage{amssymb}
\usepackage[shortlabels]{enumitem}

\newtheorem{prop}{Proposition}[section]
\newtheorem{theorem}[prop]{Theorem}
\newtheorem{rem}[prop]{Remark}
\newtheorem{lemma}[prop]{Lemma}
\newtheorem{cor}[prop]{Corollary}

\theoremstyle{definition}
\newtheorem{defi}[prop]{Definition}

\def\cH{\mathcal{H}}
\def\cK{\mathcal{A}}
\def\cB{\mathcal{B}}

\def\cD{\mathcal{D}}

\def\cF{\mathcal{F}}

\def\cK{\mathcal{K}}
\def\cL{\mathcal{L}}
\def\cM{\mathcal{M}}

\def\cO{\mathcal{O}}
\def\cU{\mathcal{U}}
\def\cX{\mathcal{X}}
\def\cY{\mathcal{Y}}
\def\R{\mathbb R}

\def\C{\mathbb C}
\def\N{\mathbb N}

\def\ux{\underline{x}}
\def\uy{\underline{y}}
\def\uz{\underline{z}}
\def\u0{\underline{0}}
\def\uomega{\underline{\omega}}

\def\b0{\mbox{\boldmath $0$}}

\def\gl2{{\rm GL}_2(\R)}
\def\gln{{\rm GL}_n(\R)}

\title[The Affine Group of the Plane]{Harmonic Analysis on the Affine Group of 
the Plane}

\author{Raja Milad}
\address{Department of Mathematics \& Statistics, Dalhousie University}
\email{raja.ma.milad@dal.ca}

\author{Keith F. Taylor}
\address{Department of Mathematics \& Statistics, Dalhousie University}
\email{keith.taylor@dal.ca}

\begin{document}

\begin{abstract}
For any natural number $n$, the group $G_n$ of
all invertible affine transformations of 
$n$-dimensional Euclidean space has, up to equivalence, just one square-integrable representation and the left regular representation 
of $G_n$ is a multiple of this square-integrable
representation.
We provide a concrete realization $\sigma$ of this distinguished representation in the two dimensional 
case. We explicitly decompose the
Hilbert space $L^2(G_2)$ as a direct sum of 
left invariant closed subspaces on each of which
the left regular representation acts as a 
representation equivalent to $\sigma$.
\end{abstract}

\maketitle

\section{Introduction} 
The purpose of this paper is to provide concrete
tools, analogous to those provided by the 
Peter--Weyl Theorem for compact groups, for the analysis of square-integrable
functions on the group $G_2$ of all invertible affine
transformations of the Euclidean plane, equipped
with its left Haar measure. This group
is isomorphic, as a topological group, to the
semidirect of the vector group $\R^2$ with 
$\gl2$, the
group of invertible $2\times2$ real matrices, acting on $\R^2$ in the natural manner.
Although $G_2$ is not compact, it shares an important feature with compact groups in that its left
regular representation decomposes as a direct sum
of irreducible representations. Indeed,
for any positive integer
$n$, if $G_n$
denotes the group of invertible affine transformations of $\R^n$, then the left regular representation 
of $G_n$ is an infinite multiple of a single
distinguished irreducible representation. It is
not so difficult to prove this fact about $G_n$
using general results from Mackey's theory of
induced representations as introduced in \cite{Mac}.
However, since we are not aware of an explicit statement of
this feature of $G_n$ in the literature, we
provide it in Theorem \ref{mult_G_n} along with a proof. For the affine group of the line, $G_1$,
it has been known for a long time that its left
regular representation is a multiple of a single
square--integrable representation (see Example 11.3 of
\cite{KL}, for example) and it is relatively routine
to work out various equivalent descriptions of the
distinguished representation. Most of the
original work in this current paper consists of
finding an explicit description of the one
irreducible representation of $G_2$ that occurs in
its left regular representation and to give an 
explicit way to write the left regular representation
of $G_2$ as an infinite multiple of this one
irreducible representation.

For the
rest of this introduction, please refer to Section 2
for any unfamiliar or ambiguous notational conventions. Also, in Section 2, we review
known results on square--integrable representations;
in particular, we recall the Duflo--Moore
Theorem and the generalization of the Peter-Weyl
Theorem that holds for any group whose regular
representation is a direct sum of irreducible 
representations. 

Section 3 is devoted to induced representations. 
We recall one of the equivalent ways of defining 
${\rm ind}_H^G\pi$, where $H$ is a closed
subgroup of a locally compact group $G$ and
$\pi$ is a unitary representation of $H$. In 
the case that there exists a closed subgroup
$K$ of $G$ that is complementary to $H$ (see
Section 3 for the exact definition), we 
show how ${\rm ind}_H^G\pi$ is equivalent to
a representation realized on the Hilbert space
$L^2(K,\cH_\pi)$ (see Proposition \ref{KH_ind}). In that process, we correct an
error in Example 2.29 of \cite{KT}.

In Section 4, we consider $G_n$, for general $n$.
If $N$ denotes the closed abelian normal subgroup
of pure translations and if $\chi_{\uomega}$ is
any nontrivial character of $N$, we construct
a representation $\pi^{\uomega}$, which acts on
the Hilbert space $L^2\big(\gln\big)$ and is equivalent
to ${\rm ind}_N^{G_n}\chi_{\uomega}$. In fact, 
 the $\pi^{\uomega}$ are all
mutually equivalent, 
for $\uomega\in\widehat{\R^n}
\setminus\{\u0\}$. 
To make notation as simple
as possible, let $\uomega_0=(1,0\cdots,0)$. In
Theorem \ref{decomp_GLn_1}, we give an explicit
way of writing $L^2(G_n)$ as a direct sum of
left invariant closed subspaces so that the
restriction of the left regular representation of
$G_n$ to each of these subspaces is equivalent to
$\pi^{\uomega_0}$. When $n=1$, $\uomega_0=1$. The
representation
$\pi^1$ turns out to be the unique 
infinite dimensional 
irreducible representation of $G_1$. We show how
$\pi^1$ is equivalent to the well--known 
{\em natural} representation of $G_1$ on $L^2(\R)$.
We note that the 
fact that the natural representation of
$G_1$ on $L^2(\R)$ is irreducible, even square--integrable, is
closely tied to the reproducing property of the
continuous wavelet transform on $\R$ (see \cite{GMP}).
However, for $n>1$, $\pi^{\uomega_0}$ is reducible.
Using mathematical induction and the Induction in
Stages Theorem, we show that $\pi^{\uomega_0}$ can be
further decomposed when $n>1$ 
and prove that there exists
an irreducible representation $\sigma_n$ of $G_n$
such that the left regular representation of
$G_n$ is equivalent to a countably infinite
multiple of $\sigma_n$ (Theorem \ref{mult_G_n}).
In particular, $\sigma_n$ is square--integrable.
Moreover, up to equivalence, there must be
only one square--integrable representation of
$G_n$, for each $n$.
This is useful to know in a theoretical sense,
but the real nature of $\sigma_n$, for $n>1$, is
obscured by nested operations of inducing, direct
sums, and taking tensor products. For general $n$,
it is a challenging problem to find a realization 
of this abstract $\sigma_n$ that acts on a concrete 
Hilbert space.
We succeed in Section 5 for the case of $n=2$. This
provides a bridge between the soft methods of 
abstract harmonic analysis and the tools of
classical harmonic analysis for study of functions
on the affine group of the plane. We will
use $\sigma$ instead of $\sigma_2$ to denote
the concrete realization of the unique 
square--integrable representation of $G_2$.

One key to the success when $n=2$ 
is a factorization of $\gl2$ as
a product of two closed subgroups. They are
\[
K_0=\left\{\begin{pmatrix}
s & -t\\
t & s
\end{pmatrix}: s,t\in\R, s^2+t^2>0\right\} 
\text{ and }
H_{(1,0)}=\left\{\begin{pmatrix}
1 & 0\\
u & v
\end{pmatrix}: u,v\in\R, v\neq0\right\}.
\]
 Direct
calculations show that $\gl2=K_0H_{(1,0)}$, 
$K_0\cap H_{(1,0)}=\{{\rm id}\}$, where id denotes
the identity matrix, and 
$(M,C)\to MC$ is a homeomorphism of 
$K_0\times H_{(1,0)}$ with $\gl2$. This immediately
provides a similar decomposition of $G_2$. Let 
$K=\{[\u0,M]:M\in K_0\}$ and 
$H=\{[\ux,C]:\ux\in\R^2,C\in H_{(1,0)}\}$. Then
$K\cap H=\{[\u0,{\rm id}]\}$ and $G_2=KH$.
Note that
$\cO=\{(1,0)A:A\in\gl2\}=\widehat{\R^2}\setminus
\{(0,0)\}$ is the nontrivial orbit in 
$\widehat{\R^2}$ under the action of $G_2$ and
$H$ is the stability subgroup of 
$\uomega_0=(1,0)$ under that action. 
Because $K$ is complementary to
$H$, we can identify $K$ with both
the quotient space $G_2/H$ and the
orbit $\cO$. To make our formula for $\sigma$
readable, we define a homeomorphism $\gamma:\cO\to
K_0$ by $\gamma(\uomega)=\frac{1}{\|\uomega\|^2}
\begin{pmatrix}
\omega_1 & -\omega_2\\
\omega_2 & \omega_1
\end{pmatrix}$, for $\uomega\in\cO$, where
$\|\uomega\|$ is the Euclidean norm of $\uomega$.
We also define two rational functions of
six real variables, expressed as functions of
$\uomega\in\cO$ and $A\in\gl2$. They are denoted
$u_{\uomega,A}$ and $v_{\uomega,A}$. The precise
expressions for $u_{\uomega,A}$ and $v_{\uomega,A}$
are given in Proposition \ref{u_omegaA_v_omegaA}.
We can now present a detailed expression for
$\sigma$ realized on a concrete Hilbert space.

The Hilbert space of $\sigma$ is $L^2\big(K,
L^2(\R^*)\big)$, where $\R^*=\R\setminus\{0\}$
viewed as a group under multiplication of real numbers. Both $K$ and $\R^*$ are equipped with
their respective Haar measures. We note that
$K=\{[\u0,\gamma(\uomega)]:\uomega\in\cO\}$.
For $F\in L^2\big(K,L^2(\R^*)\big)$, 
$[\ux,A]\in G_2$, and a.e. $\uomega\in\cO$,
$\sigma[\ux,A]F[\u0,\gamma(\uomega)]$ is the
element of $L^2(\R^*)$ given by
\begin{equation}\label{sigma_pointwise_formula_1}
\big(\sigma[\ux,A]F[\u0,\gamma(\uomega)]\big)(t) = 
\textstyle\frac{|\det(A)|^{1/2}\|\uomega\|}{\|\uomega A\|}
e^{2\pi i(\uomega\ux+t^{-1}u_{\uomega,A})}
\big(F[\u0,\gamma(\uomega A)]\big)
(v_{\uomega,A}^{-1}t),
\end{equation}
for a.e. $t\in\R^*$.
We note that Mackey theory (see \cite{Mac}, \cite{Fol},
or \cite{KT}) tells us that
\[
\widehat{G_2}=\{[\sigma]\}\cup\left\{\tilde{[\pi]}:
[\pi]\in\widehat{\gl2}\right\}.
\]
Here $[\tau]$ denotes the unitary equivalence class
of a representation $\tau$ and, for a representation
$\pi$ of $\gl2$, the lift of $\pi$ to $G_2$ is
denoted $\tilde{\pi}$. A parametrization of 
$\widehat{\gl2}$ is complex (see, for example, \cite{Vog}), but only $\sigma$ is needed for
the analysis of square--integrable functions on $G_2$.

We show that $\sigma$ is square--integrable by
a multi--step process that concludes with the following
core
result (Theorem \ref{V_E_isometry}). 
Let $g\in L^2(\R^*)$ satisfy $\int_{\R^*}\big|
|\nu|^{1/2}g(\nu)\big|^2d\mu_{\R^*}(\nu)=1$ and let
$\eta\in L^2(\widehat{\R^2})$ satisfy
$\|\eta\|_{_{L^2(\widehat{\R^2})}}=1$.
Let $E\in  L^2\big(K,L^2(\R^*)\big)$ be defined
as 
\[
E[\u0,L](\nu)=
\frac{\overline{\eta}((1,0)L^{-1})}{|\det(L)|}\,g(\nu),
\text{ for a.e. } [\u0,L]\in K \text{ and a.e. }
\nu\in\R^*.
\]
 Define $V_EF[\ux,A]=
\langle F,
\sigma[\ux,A]E\rangle_{_{L^2(K,L^2(\R^*))}}$, for
$[\ux,A]\in G_2$ and 
$F \in L^2\big(K,L^2(\R^*)\big)$. Then $V_E$ is
a linear isometry of $L^2\big(K,L^2(\R^*)\big)$
into $L^2(G_2)$ that intertwines $\sigma$ with
the left regular representation of 
$G_2$. In particular, $\sigma$ is a square-integrable representation of $G_2$.

In Theorem \ref{V_E_ij}, we present a doubly 
indexed family $\{\cM_{i,j}:(i,j)\in I\times J\}$
of closed, left invariant, subspaces
of $L^2(G_2)$ such that 
$L^2(G_2)=\sum^{\oplus}_{(i,j)\in I\times J}\cM_{i,j}$ and the restriction of the left regular
representation of $G_2$ to each $\cM_{i,j}$ is
unitarily equivalent to $\sigma$, showing that
the regular representation is supported by the
singleton $\{[\sigma]\}$ in $\widehat{G_2}$. 
This also enables us to
easily formulate what the abstract Plancherel Theorem
looks like 
for the group $G_2$ and we do this in Proposition 
\ref{Plancherel_G_2}. 
We end with a section of concluding
remarks.

\section{Notation and background}
In this section, we introduce our notational conventions as well as those properties of 
square-integrable representations that we need.
Let $n\in\N$. 
\begin{itemize}
\item $\R$ denotes the field of real numbers. Under
addition as operation, $\R$ is a group.
\item $\R^*=\R\setminus\{0\}$, which is a group under multiplication.
\item $\R^+$ is the subgroup of $\R^*$ consisting of 
positive real numbers.
\item $\R^n=\left\{\ux=\begin{pmatrix}
x_1\\
\vdots\\
x_n
\end{pmatrix}: x_1, \cdots, x_n\in\R\right\}$,
$\widehat{\R^n}=\{\uomega=(\omega_1,\cdots,\omega_n)
: \omega_1,\cdots,\omega_n\in\R\}$.
\item The symbol $\u0$ will be used for either
a column vector or row vector of zeroes; whichever
fits in the context.
\item For $f\in L^1(\R^n)$, the Fourier transform of
$f$ is $\widehat{f}(\uomega)=
\int_{\R^n}f(\ux)e^{2\pi i\uomega\ux}d\ux$, for
all $\uomega\in\widehat{\R^n}$.
\item $\cF:L^2(\R^n)\to L^2\big(\widehat{\R^n}\big)$
is the unitary such that $\cF f=\widehat{f}$, for 
all $f\in L^2(\R^n)\cap L^1(\R^n)$.
\item $\gln$ denotes the general linear group of 
invertible $n\times n$ real matrices.
\item For $\ux\in\R^n$ and $A\in\gln$, $[\ux,A]$
is the affine transformation $\uz\to\ux+A\uz$
of $\R^n$.
\item $G_n=\{[\ux,A]:\ux\in\R^n,A\in\gln\}$.
\end{itemize}
For $[\ux,A],[\uy,B]\in G_n$, the composition,
denoted $[\ux,A][\uy,B]$, is the affine transformation $[\ux+A\uy,AB]$. In fact, 
$G_n$ forms a group when equipped with composition
as a product. The identity in $G_n$ is 
$[\u0,{\rm id}]$, where ${\rm id}$ denotes the
identity $n\times n$ matrix. If $\gln$ is considered as an open subset of $\R^{n^2}$, then $\R^n\times
\gln$ is a locally compact Hausdorff space with the
product topology and we move this topology to $G_n$
with the bijection $(\ux,A)\to[\ux,A]$. With this
topology, $G_n$ is a locally compact group which can be viewed as the semidirect product $\R^n\rtimes
\gln$ (see the discussion following Example 1.5
in \cite{KT} for basic information on semidirect products). For any locally compact group $G$, we
use the following notation:
\begin{itemize}
\item $\mu_G$ denotes left Haar measure on $G$. It
is a nonzero Radon measure with the property that
$\mu_G(xE)=\mu_G(E)$, for any Borel $E\subseteq G$
and $x\in G$. Left Haar measure is unique up to
multiplication by a positive constant.
\item Integrals with respect to $\mu_G$ are denoted
$\int_Gf\,d\mu_G$ or $\int_Gf(x)\,d\mu_G(x)$, for
any function $f$ on $G$ for which the integral has
meaning. Thus, $\int_Gf(yx)\,d\mu_G(x)=
\int_Gf(x)\,d\mu_G(x)$, for any $y\in G$.
\item The {\em modular function} on $G$ is the 
continuous homorphism $\Delta_G:G\to\R^+$ that 
satisfies $\Delta_G(y)\int_Gf(xy)\,d\mu_G(x)=
\int_Gf(x)\,d\mu_G(x)$, for any $y\in G$ and any
$f$ where the integral makes sense. We call 
$G$ {\em unimodular} if $\Delta_G$ is identically 1
on $G$; otherwise $G$ is {\em nonunimodular}.
\item For $1\leq p\leq\infty$, $L^p(G)$ denotes 
the usual Lebesgue space $L^p(G,\cB_G,\mu_G)$,
where $\cB_G$ denotes the $\sigma$-algebra of 
Borel subsets of $G$.
\item $L^2(G)$ is a Hilbert space with inner
product given by
$\langle f,g\rangle_{L^2(G)}=
\int_Gf(x)\overline{g(x)}\,
d\mu_G(x)$, for all $f,g\in L^2(G)$.
\item $C_b(G)$ is the space of bounded continuous
$\C$-valued functions on $G$, a Banach space when
equipped with the uniform norm. 
\item $C_c(G)$ denotes the space of continuous
$\C$-valued functions of compact support on $G$.
\end{itemize}
For a particular locally compact group $G$, we often
identify $\mu_G$ by giving a formula for
$\int_Gf\,d\mu_G$ for any $f\in C_c(G)$. This
uniquely identifies the left Haar measure, and the 
normalization we are using, by the Riesz Representation Theorem. For example, the left
Haar measure on $\gln$ is described as follows 
(see \cite{HR}).
For $A\in\gln$, let $a_{ij}$ denote the $(i,j)$
entry of $A$. Then, 
\begin{equation}\label{Haar_GL_n}
\int_{\gln}f\,d\mu_{\gln}=
\int_{-\infty}^\infty\cdots\int_{-\infty}^\infty
f\begin{pmatrix}
a_{11} & \cdots & a_{1n}\\
\vdots & \ddots & \vdots\\
a_{n1} & \cdots & a_{nn}
\end{pmatrix}
\frac{da_{nn}\cdots da_{11}}{|\det(A)|^n},
\end{equation}
for any $f\in C_c(\gln)$.
In particular, 
\[
\int_{\gl2}f\,d\mu_{\gl2}=
\int_{-\infty}^\infty\int_{-\infty}^\infty
\int_{-\infty}^\infty\int_{-\infty}^\infty
f\begin{pmatrix}
w & x\\
y & z
\end{pmatrix}\frac{dz\,dy\,dx\,dw}{(wz-xy)^2},
\]
for any $f\in C_c(\gl2)$. Note that $\mu_{\gln}$
is also right invariant, so $\gln$ is a
unimodular group.

For a semidirect product group such as $G_n$,
the left Haar measure can be computed as described
on page 9 of \cite{KT}. This gives
\begin{equation}\label{Haar_G_n}
\int_{G_n}f\,d\mu_{G_n}=\int_{\gln}\int_{\R^n}
f[\ux,A]\,\frac{d\ux\,d\mu_{\gln}(A)}{|\det(A)|},
\text{ for all }f\in C_c(G_n).
\end{equation}
The modular function of $G_n$ is given by
$\Delta_{G_n}[\ux,A]=|\det(A)|^{-1}$, for 
$[\ux,A]\in G_n$.

For the theory of Hilbert spaces and operators
on Hilbert spaces, a good reference is \cite{KR}.
Let $\cH$ be a Hilbert space. 
\begin{itemize}
\item The norm and inner product
on $\cH$ are denoted by $\|\cdot\|_{\cH}$ and 
$\langle\cdot,\cdot
\rangle_{\cH}$, respectively. 
\item The Banach $*$-algebra of 
bounded linear operators on $\cH$ is denoted
$\cB(\cH)$. 
\item The group of all unitary
operators on $\cH$ is $\cU(\cH)$. 
\item $A\in\cB(\cH)$ is called a {\em Hilbert--Schmidt operator} on $\cH$ if 
$\sum_{j\in J}\|A\xi_j\|_{\cH}^{\,\,2}<\infty$, for
any orthonormal basis $\{\xi_j:j\in J\}$ of $\cH$.
The sum $\sum_{j\in J}\|A\xi_j\|_{\cH}^{\,\,2}$ is
independent of the basis.
\item Let $\cB_2(\cH)$ denote the set of all
Hilbert--Schmidt operators on $\cH$. This is a 
Hilbert space with norm given by
$\|A\|_{\rm HS}=\left(
\sum_{j\in J}\|A\xi_j\|_{\cH}^{\,\,2}\right)^{1/2}$,
for any orthonormal basis 
$\{\xi_j:j\in J\}$ of $\cH$.
\item If $(X,\mu)$ is a $\sigma$--finite 
measure space and $g:X\to\C$
is measurable, for any $f\in L^2(X,\mu)$, let
$M_gf=gf$. If $g$ is bounded, then 
$M_g\in\cB\big(L^2(X,\mu)\big)$. In general, $M_g$
is a not necessarily bounded operator on 
$L^2(X,\mu)$ (see Example 2.7.2 of \cite{KR}).
\end{itemize}

Let $G$ be a locally compact group. Let $\cH_\pi$
be a Hilbert space and let $\pi:G\to\cU(\cH_\pi)$ be
a group homomorphism. We call $\pi$, more formally
the pair $(\cH_\pi,\pi)$, a {\em unitary
representation} of $G$ if, for every $\xi,\eta
\in\cH_\pi$, the map $x\to\langle\xi,\pi(x)\eta
\rangle_{\cH_\pi}$ is a continuous function on $G$.
When no confusion can arise, we will simply write
representation of $G$ to mean a unitary representation of $G$. The {\em left regular
representation} of $G$ is denoted $\lambda_G$.
The Hilbert space of $\lambda_G$ is $L^2(G)$ and,
for each $x\in G$, $\lambda_G(x)f(y)=
f(x^{-1}y)$, for $\mu_G$-almost every
 $y\in G$ and any $f\in L^2(G)$.
For a given representation
$\pi$ of $G$, we have the following notation and concepts:
\begin{itemize}
\item For any $f\in L^1(G)$, there exists 
$\pi(f)\in\cB(\cH_\pi)$ that satisfies
\[
\langle\pi(f)\xi,\eta\rangle=
\int_G f(x)\langle\pi(x)\xi,\eta\rangle d\mu_G(x),
\text{ for all }\xi,\eta\in\cH_\pi.
\]
Then $\pi:L^1(G)\to\cB(\cH_\pi)$ is a non-degenerate
$*$-homomorphism.
\item For the left regular representation,
$\lambda_G(f)h=f*h$, for $h\in L^2(G)$ and
$f\in L^1(G)$.
\item Suppose $\sigma$ is another representation
of $G$ and $A:\cH_\sigma\to\cH_\pi$ is a bounded
linear operator. We say that $A$ {\em intertwines}
$\sigma$ and $\pi$ if $A\sigma(x)=\pi(x)A$, for all
$x\in G$.
\item If $\sigma$ is another representation of 
$G$, we say $\sigma$ is {\em equivalent} to $\pi$ if
there exists a unitary map $U:\cH_\sigma\to\cH_\pi$
that intertwines $\sigma$ and $\pi$. Then $U^{-1}$
intertwines $\pi$ and $\sigma$.
 We write
$\sigma\sim\pi$ when $\sigma$ and $\pi$ are
equivalent. Then $\sim$ is an equivalence 
relation on any set of representations of $G$.
\item A closed subspace $\cK$ of $\cH_\pi$ is 
{\em $\pi$-invariant} if $\xi\in\cK$ implies 
$\pi(x)\xi\in\cK$, for all $x\in G$.
\item If $\cK$ is a $\pi$-invariant closed 
subspace of $\cH_\pi$, let $\pi_{\cK}$ denote
the representation whose Hilbert space is $\cK$ and
satisfies $\pi_{\cK}(x)=\pi(x)|_{\cK}$, for all
$x\in G$. We call $\pi_{\cK}$ a {\em subrepresentation} of $\pi$.
\item If $J$ is an index set and $\pi_j$ is a
representation of $G$ on $\cH_j$, for each $j\in J$,
then $\oplus_{j\in J}\pi_j$ denotes the
representation of $G$ acting on 
$\underline{\xi}=(\xi_j)_{j\in J}\in 
\oplus_{j\in J}\cH_j$ by 
\[
\big(\oplus_{j\in J}\pi_j\big)(x)\underline{\xi}=
\big(\pi_j(x)\xi_j\big)_{j\in J}, \text{ for }
x\in G.
\]
\item If $J$ is an index set with cardinality $m$,
suppose
$\pi_j=\pi$ and $\cH_j=\cH_\pi$, for all $j\in J$,
then $\oplus_{j\in J}\pi_j$ may be denoted
$m\cdot\pi$. If $\sigma$ is a representation of $G$
such that $\sigma\sim m\cdot\pi$, then $\sigma$
may be called a {\em multiple of} 
$\pi$.
\item If $\{0\}$ and $\cH_\pi$ are the only
$\pi$-invariant closed subspaces of $\cH_\pi$,
then $\pi$ is called an {\em irreducible 
representation} of $G$.
\item For each $\eta\in\cH_\pi$, define the linear
map $V_\eta:\cH_\pi\to C_b(G)$ by, for each 
$\xi\in\cH_\pi$,
\[
V_\eta\xi(x)=\langle\xi,\pi(x)\eta\rangle_{\cH_\pi},
\text{ for all }x\in G.
\]
\item $\pi$ is called a {\em square-integrable
representation} if $\pi$ is irreducible and
the exists a nonzero $\eta\in\cH_\pi$ such that
$V_\eta\eta\in L^2(G)$.
\end{itemize}
Examples of square-integrable representations 
arise in an elementary manner for many 
groups of affine transformations. 
Let $K$ be a closed subgroup of
$\gln$ and let $G=\R^n\rtimes K=\{[\ux,A]:\ux\in\R^n, 
A\in K\}$, a closed subgroup of $G_n$. For
$\uomega\in\widehat{\R^n}$, The $K$-orbit of
$\uomega$ is $\uomega K=\{\uomega A:A\in K\}$ and
the stabilizer is $K_{\uomega}=
\{A\in K:\uomega A=\uomega\}$. The {\em natural 
representation} $\rho$ of $G$ on $L^2(\R^n)$ 
is given by, for $[\ux,A]\in G$ and $f\in L^2(\R^n)$,
\begin{equation}\label{natural}
\rho[\ux,A]f(\uy)=|\det(A)|^{-1/2}f\left(
A^{-1}(\uy-\ux)\right), \text{ for a.e. }
\uy\in\R^n.
\end{equation}
If $\cO$ is an open subset of $\widehat{\R^n}$, let
$\cH^2_{\cO}$ denote the closed
subspace of $L^2(\R^n)$
consisting of functions whose Fourier transform is
supported on $\cO$. If 
$\cO K=\{\uomega A:\uomega\in\cO, A\in K\}=\cO$,
then $\cH^2_{\cO}$ is a $\rho$-invariant subspace
of $L^2(\R^n)$. Let $\rho_{\cO}$ denote the 
corresponding subrepresentation of $\rho$.
If $\cO$ is a single $K$-orbit that is free,
in the sense that $K_{\uomega}=\{{\rm id}\}$, for
any $\uomega\in\cO$, then $\rho_{\cO}$ was shown to
be square-integrable in \cite{BT}.
The following more general 
proposition was established for open
$K$-orbits with compact stability subgroups in
\cite{Fuh1}.
\begin{prop}\label{open_orbit_sq_int}
Let $K$ be a closed subgroup of $\gln$ and let
$G=\R^n\rtimes K$. Suppose there exists an
$\uomega\in\widehat{\R^n}$ such that $\cO=\uomega K$
is open. Then 
$\rho_{\cO}$ is a square-integrable representation
of $G$ if and only if $K_{\uomega}$ is compact.
\end{prop}

Our primary reference for the properties of square-integrable 
representations is \cite{DM}. The following theorem
is formulated from Theorem 3 of \cite{DM} and 
the preliminary results in \cite{DM}. Also see Theorem 2.25 of \cite{Fuh}.
\begin{theorem}\label{Du_Mo} {\bf [Duflo-Moore]}
Let $\pi$ be a square-integrable representation
of a locally compact group. Let $\cD_\pi=\{\eta:
V_\eta\eta\in L^2(G)\}$. Then\\
\indent {\rm (a)} $\cD_\pi$ is a dense
linear subspace of $\cH_\pi$,\\
\indent {\rm (b)} if $\eta\in\cD_\pi$,
then $V_\eta\xi\in L^2(G)$, for all $\xi\in\cH_\pi$,
\\
\indent {\rm (c)} there exists a nonzero positive
selfadjoint operator $C_\pi$ in $\cH_\pi$
with domain $\cD_\pi$ satisfying
\[
\pi(x)C_\pi\pi(x)^*=\Delta_G(x)^{1/2}C_\pi, 
\text{ for all } x\in G,
\]
\indent {\rm (d)} for any $\xi_1,\xi_2\in\cH_\pi$
and $\eta_1,\eta_2\in\cD_\pi$,
\[
\langle V_{\eta_1}\xi_1,
V_{\eta_2}\xi_2\rangle_{L^2(G)}=
\langle\xi_1,\xi_2\rangle_{\cH_\pi}
\langle C_\pi\eta_2,C_\pi\eta_1\rangle_{\cH_\pi}
\]
Moreover, if $T$ is any densely defined nonzero
positive selfadjoint operator in $\cH_\pi$ that
satisfies the identity $\pi(x)T\pi(x)^*=\Delta_G(x)^{1/2}T$,
for all $x\in G$, then the domain of $T$ is $\cD_\pi$ 
and there exists a constant 
$r> 0$ such that $T=rC_\pi$.
\end{theorem} 
\begin{defi}
Let $\pi$ be a square-integrable representation of
a locally compact group $G$. The operator 
$C_\pi$ named in Theorem \ref{Du_Mo} is called the
{\em Duflo-Moore operator} of $\pi$.
\end{defi}
Let $\pi$ be a square-integrable representation of
a locally compact group $G$. For any 
$\eta\in\cD_\pi$, by (b) and (d) of Theorem \ref{Du_Mo},
$\|V_\eta\xi\|_{L^2(G)}=\|C_\pi\eta\|_{\cH_\pi}
\|\xi\|_{\cH_\pi}$, for all $\xi\in\cH_\pi$.
Thus, $\cK^\pi_\eta=\{V_\eta\xi:\xi\in\cH_\pi\}$ is
a closed subspace of $L^2(G)$ and $V_\eta:\cH_\pi
\to \cK^\pi_\eta\subseteq L^2(G)$ is a bounded linear map.
Moreover, if $\|C_\pi\eta\|_{\cH_\pi}=1$, then
$V_\eta$ is an isometry into $L^2(G)$. In fact,
(d) of Theorem \ref{Du_Mo} shows that $V_\eta$ is
a unitary map when considered as a map of $\cH_\pi$
onto $\cK^\pi_\eta$. Note that, for $x\in G$ and 
$\xi\in\cH_\pi$,
\begin{equation}\label{intertwine_V_eta_lambda}
\lambda_G(x)V_\eta\xi(y)=V_\eta\xi(x^{-1}y)=
\langle\xi,\pi(x^{-1}y)\eta\rangle_{\cH_\pi}=
\langle\pi(x)\xi,\pi(y)\eta\rangle_{\cH_\pi}=
V_\eta\pi(x)\xi(y),
\end{equation}
for all $y\in G$. Thus, $\cK^\pi_\eta$ is a 
$\lambda_G$-invariant subspace of $L^2(G)$ and
$V_\eta$ intertwines $\pi$ with $\lambda_G$. 
This means $\pi$ is equivalent to the restriction 
of $\lambda_G$ to $\cK^\pi_\eta$. If $\eta_1,\eta_2
\in\cD_\pi$ are such that $C_\pi\eta_1\perp 
C_\pi\eta_2$, then $\cK^\pi_{\eta_1}\perp
\cK^\pi_{\eta_2}$ by Theorem \ref{Du_Mo}(d).
Following \cite{Fuh}, let $L^2_\pi(G)$ denote
the smallest closed subspace of $L^2(G)$ containing
$\cup_{\eta\in\cD_\pi}\cK^\pi_\eta$. A Gramm--Schmidt
process shows that, when $G$ is separable, there
exists a countable set $\{\eta_j:j\in J\}$ in 
$\cD_\pi$ such that $\{C_\pi\eta_j:j\in J\}$ is
an orthonormal basis of $\cH_\pi$. Then
$L^2_\pi(G)=\oplus_{j\in J}
\cK^\pi_{\eta_j}$ (see Theorem 2.33(c) of \cite{Fuh}). Note that
the cardinality of $J$ is $\dim(\cH_\pi)$.

Suppose $\pi$ and $\pi'$ are both square-integrable
representations of $G$ and $\pi'$ is not
equivalent to $\pi$. A simple argument (see, for example, Theorem 2.33(d) of \cite{Fuh}) shows that,
for $\eta\in\cD_\pi$ and $\eta'\in\cD_{\pi'}$,
$\cK^\pi_\eta\perp \cK^{\pi'}_{\eta'}$. Thus,
$L^2_\pi(G)\perp L^2_{\pi'}(G)$.
\begin{defi}
For a locally compact group $G$, let $\widehat{G}$
denote the set of equivalence classes of irreducible
representations of $G$. Let $\widehat{G}^r$ denote
the subset of $\widehat{G}$ consisting of 
equivalence classes of
square-integrable representations.
\end{defi}
We will sometimes abuse notation and use the same symbol for an irreducible representation and
its equivalence class. A locally compact group $G$ is called an [AR]-{\em group} when its left regular
representation is a direct sum of irreducible
representations. When $G$ is an [AR]-group, we
have
\[
L^2(G)=\oplus_{\pi\in\widehat{G}^r}L^2_\pi(G).
\]
From Theorem 2.33 of \cite{Fuh} and Theorem 5.2 of
\cite{GT}, we can formulate a generalization of
the Peter-Weyl Theorem for compact groups to
[AR]-groups. 
\begin{theorem}\label{Peter_Weyl_Gen}
Let $G$ be a separable {\rm [AR]}-group. For each $\pi
\in\widehat{G}^r$, let $\{\eta^\pi_j:j\in J_\pi\}$
be a set in $\cD_\pi$ such that 
$\{C_\pi\eta^\pi_j:j\in J_\pi\}$ is an orthonormal
basis of $\cH_\pi$
and let $\{\xi^\pi_k:k\in J_\pi\}$ be any orthonormal
basis of $\cH_\pi$. Then
\[
\cup_{\pi\in\widehat{G}^r}
\left\{V_{\eta^\pi_j}\xi^\pi_k:
(j,k)\in J_\pi\times J_\pi\right\}
\]
is an orthonormal basis of $L^2(G)$.
\end{theorem}

Note that, if $G$ is compact, then 
$\widehat{G}^r=\widehat{G}$. Moreover, each $\pi\in 
\widehat{G}$ is finite dimensional and
$C_\pi=\dim(\pi)^{-1/2}I_{\cH_\pi}$, where
$\dim(\pi)=\dim(\cH_\pi)$ and $I_{\cH_\pi}$ is
the identity operator on $\cH_\pi$.

In order for the kind of basis that is guaranteed by
Theorem \ref{Peter_Weyl_Gen} to be useful when
working with a particular [AR]-group, one usually
needs explicit descriptions of the square-integrable
representations on concrete Hilbert spaces and
a precise identification of the operator $C_\pi$,
for each such $\pi$. As a result of
Proposition \ref{open_orbit_sq_int}, 
this works out well for
groups of the form $\R^n\rtimes K$, where $K$ is
a closed subgroup of $\gln$ and there
exist open $K$-orbits $\cO_1,\cdots,\cO_m$ in
$\widehat{\R^n}$, each with compact stability subgroups, such that $\cup_{j=1}^m\cO_j$ is co-null
in $\widehat{\R^n}$. However, there exist [AR]-groups
of the form $\R^n\rtimes K$ where there are open
$K$-orbits, but the associated stability subgroups
are not compact. In fact, it was shown in \cite{BT},
example (iii), that $\R^2\rtimes {\rm GL}_2^+(\R)$,
where ${\rm GL}_2^+(\R)=\{A\in \gl2:\det(A)>0\}$,
is an [AR]-group. The same is true of the full
affine group in two dimensions $G_2$. It is
worthwhile formulating a proposition collecting known
results.
\begin{prop}\label{inductive_AR}
Let $K$ be
a closed subgroup of $\gln$ such that there
exist open $K$-orbits $\cO_1,\cdots,\cO_m$ in
$\widehat{\R^n}$ with 
$\cup_{j=1}^m\cO_j$ co-null
in $\widehat{\R^n}$. For $1\leq j\leq m$, select
$\uomega_j\in\cO_j$. Then $\R^n\rtimes K$ is an
{\rm [AR]}-group if and only if $K_{\uomega_j}$ is an
{\rm [AR]}-group for $1\leq j\leq m$.
\end{prop}
\begin{proof}
This follows from the calculations on page 595 of
\cite{BT} or Corollary 11.1 of \cite{KL}.
\end{proof}

For example, in the group $G_2$, 
$K=\gl2$ and there are just
two $K$-orbits in $\widehat{\R^2}$. One orbit is
the trivial orbit $\{\u0\}$ and the other is
$\cO=\widehat{\R^2}\setminus\{\u0\}$. Pick 
$\uomega_0=(1,0)\in\cO$. Then, the stability 
subgroup is
$K_{(1,0)}=\{A\in\gl2:(1,0)A=(1,0)\}$. If
$A=\begin{pmatrix}
s & t\\
u & v
\end{pmatrix}$, then $(1,0)A=(s,t)$. So 
$A\in K_{(1,0)}$ exactly when $s=1$ and $t=0$.
Thus $K_{(1,0)}=\left\{\begin{pmatrix}
1 & 0\\
u & v
\end{pmatrix}:u\in\R,v\in\R^*\right\}$. Note
that $[u,v]\to\begin{pmatrix}
1 & 0\\
u & v
\end{pmatrix}$ is an isomorphism of $G_1$ with
$K_{(1,0)}$ and it is well-known that $G_1$ has
a regular representation that is an infinite
multiple of just one square-integrable representation
(see Example 1, Section 10, of \cite{KL}, for 
example). This is essentially the way 
$\R^2\rtimes {\rm GL}_2^+(\R)$ is shown to be an
[AR]-group in \cite{BT}.

\section{Induced Representations}
The main purpose of this study is to make the
various ingredients that appear in Theorem
\ref{Peter_Weyl_Gen}
explicit for $G_2$, the group of invertible affine transformations of the plane.
We make extensive use of induced representations. 

The
theory of induced representations for locally
compact groups in general was initially 
developed by Mackey in \cite{Mac}. The particular 
results we need are found in \cite{Fol} or \cite{KT}
and we briefly summarize them here. We introduce
a representation, equivalent to the
induced representation we need, for a particular situation that holds in all the cases that arise
in the study of $G_2$.

If $\pi$ is a unitary
representation of a closed subgroup $H$ of a 
locally compact group $G$, there are a variety of
ways (all resulting in mutually equivalent
representations) of defining the induced representation ${\rm ind}_H^G\pi$, which is a
representation of $G$. We will use the construction
presented in Section 6.1 of \cite{Fol} and 
following Proposition 2.28 of \cite{KT}.
With $G$ and $H$ fixed, a {\em rho-function} is
a Borel map $\rho:G\to[0,\infty)$ that is locally 
integable and satisfies
\begin{equation}\label{rho_function}
\rho(xh)=\frac{\Delta_H(h)}{\Delta_G(h)}\rho(x),
\text{ for }x\in G, h\in H.
\end{equation}
Let $p:G/H\to G$
be a cross-section of the $H$-cosets. That is,
$p(yH)\in yH$, for any $yH\in G/H$. 
There always exists a continuous rho-function
$\rho$ such that $\rho(x)>0$, for all $x\in G$, and
an associated regular Borel measure $\mu_\rho$ on $G/H$
such that Weil's integration formula holds. That is,
for any $f\in C_c(G)$
\begin{equation}\label{Weil}
\int_Gf(x)\rho(x)\,d\mu_G(x)=\int_{G/H}\int_H
f\big(p(\omega)h\big)\,d\mu_H(h)\,d\mu_\rho(\omega).
\end{equation}
See Theorem 1.18, Lemma 1.20, and Corollary 1.21
of \cite{KT}. If $\rho$ is a continuous and 
strictly positive rho-function on $G$, \eqref{rho_function} implies that, for $x,y\in G$, 
$\frac{\rho(x^{-1}yh)}{\rho(yh)}=
\frac{\rho(x^{-1}y)}{\rho(y)}$, for all $h\in H$. 
We need this term $\frac{\rho(x^{-1}y)}{\rho(y)}$
in the definition of ${\rm ind}_H^G\pi$. 
Fix a 
continuous, strictly positive, rho-function $\rho$ 
and let $\mu_\rho$ be the associated measure on 
$G/H$ so that \eqref{Weil} holds. The Hilbert
space of ${\rm ind}_H^G\pi$ is denoted 
$\cH_{{\rm ind}\pi}$ and consists of all equivalence
classes of weakly measurable 
functions $\xi:G\to\cH_\pi$ such that
$\xi(xh)=\pi(h^{-1})\xi(x)$, for all $h\in H$ and
$\mu_G$-almost all $x\in G$, and 
$\int_{G/H}\|\xi\big(p(\omega)\big)\|_{\cH_\pi}^{\,\,2}d\mu_\rho(\omega)<\infty$. Note that the latter 
condition does not depend on the choice of the
cross-section $p$. For $\xi_1,\xi_2\in
\cH_{{\rm ind}\pi}$, the inner product is given
by
\[
\langle\xi_1,\xi_2\rangle_{\cH_{{\rm ind}\pi}}=
\int_{G/H}\langle\xi_1\big(p(\omega)\big),
\xi_2\big(p(\omega)\big)\rangle_{\cH_\pi}
d\mu_\rho(\omega).
\]
Then $\cH_{{\rm ind}\pi}$ is a Hilbert space. For
$x\in G$ and $\xi\in\cH_{{\rm ind}\pi}$, let
\begin{equation}\label{ind_formula}
{\rm ind}_H^G\pi(x)\xi(y)=\textstyle
\left[\frac{\rho(x^{-1}y)}{\rho(y)}\right]^{1/2}
\xi(x^{-1}y), \text{ for } \mu_G-{\rm a.e.} \,
y\in G.
\end{equation}
This defines a unitary operator ${\rm ind}_H^G\pi(x)$
on $\cH_{{\rm ind}\pi}$, for each $x\in G$, and
a representation ${\rm ind}_H^G\pi$ of $G$. See
\cite{KT} or \cite{Fol} for more details. Using
a different rho-function in the definition results
in an equivalent representation. In some 
situations, 
one can show that ${\rm ind}_H^G\pi$
is equivalent to a representation acting on a
more natural Hilbert space of functions
on a space that has a natural $G$-action. One
 situation that is useful to us 
is treated in Example 2.29 of \cite{KT}. However,
there is an error in the definition of the rho-function in the last line of page 74 of \cite{KT},
so we will point out the details necessary to
correct the formula given there.

Suppose $H$ is a closed subgroup of $G$ that is
{\em complemented} in the sense that there exists
a closed subgroup $K$ of $G$ such that $K\cap H=
\{e\}$, where $e$ is the identity element of
$G$, and the map $(k,h)\to kh$ is a homeomorphism
of $K\times H$ onto $G$. When this holds, 
the restriction $q|_K$ of the quotient map $q$ to
$K$ is a homeomorphism of $K$ with $G/H$. For each
$x\in G$, there exist unique $k_x\in K$ and $h_x
\in H$ such that $x=k_xh_x$. We take the cross-section $p:G/H\to G$ to be given by $p(xH)=k_x$, for 
any $xH\in G/H$. Then $p$ is a homeomorphism of
$G/H$ with its image $K$. The left action of 
$G$ on $G/H$ is transferred to a left action of $G$
on $K$ by this homeomorphism. That is, for $x\in G$ 
and $k\in K$, $x\cdot k=k_{xk}$. 
Define $\rho:G\to\R^+$
by $\rho(x)=\frac{\Delta_H(h_x)}{\Delta_G(h_x)}$,
for all $x\in G$. For $x\in G$ and $h\in H$,
$h_{xh}=h_xh$, so $\rho(xh)=
\frac{\Delta_H(h)}{\Delta_G(h)}\rho(x)$. Thus,
$\rho$ is a rho-function on $G$ that is obviously
continuous and positive. Since $h_{kx}=h_x$, 
we have $\rho(kx)=\rho(x)$, for 
$k\in K$ and $x\in G$. 

The measure $\mu_\rho$ that 
is associated to $\rho$ so that \eqref{Weil} holds
can be moved to $K$ by the homeomorphism $q|_K$.
That is, let $\tilde{\mu}_\rho(E)=
\mu_\rho\big(q(E)\big)$, 
for any Borel $E\subseteq K$, and
$\int_K\psi(k)\,d\tilde{\mu}_\rho(k)=
\int_{G/H}\psi\big(p(\omega)\big)\,d\mu_\rho(\omega)$, for any $\psi\in C_c(K)$. It turns out that
$\tilde{\mu}_\rho$ is left Haar measure on $K$.
To see this, let $\varphi\in C_c(K)$ and $\ell
\in K$. Let $\varphi'\in C_c(G/H)$ be defined by
$\varphi'(kH)=\varphi(k)$, for all $k\in K$. By
Proposition 1.9 of \cite{KT}, there
exists an $f\in C_c(G)$ so that
$\varphi(k)=\varphi'(kH)=\int_Hf(kh)\,d\mu_H(h)$,
for all $k\in K$. Now, using \eqref{Weil}, we have
\[
\int_K\varphi(\ell k)\,d\tilde{\mu}_\rho(k)=
\int_{G/H}\int_Hf\big(\ell p(\omega)h\big)\,d\mu_H(h)\,d\mu_\rho(\omega)
=\int_Gf(\ell x)\rho(x)\,d\mu_G(x).
\]
But $\rho(\ell x)=\rho(x)$, for every $x\in G$, and
left invariance of $\mu_G$ yields
$\int_K\varphi(\ell k)\,d\tilde{\mu}_\rho(k)=
\int_Gf(x)\rho(x)\,d\mu_G(x)=
\int_K\varphi(k)\,d\tilde{\mu}_\rho(k)$. Since
$\varphi\in C_c(K)$ and $\ell\in K$ were
arbitrary, $\tilde{\mu}_\rho$ must be  
left Haar measure of $K$. This normalizes the
choices of left Haar measures on $G$, $H$, and $K$
so that
\begin{equation}
\textstyle
\int_K\int_Hf(kh)\,d\mu_H(h)\,d\mu_K(k)=
\int_Gf(x)\frac{\Delta_H(h_x)}{\Delta_G(h_x)}\,
d\mu_G(x), \text{ for all }f\in C_c(G).
\end{equation}

Let $L^2(K,\cH_\pi)$ denote the space of 
$\cH_\pi$-valued weakly measurable functions $F$ on
$K$ such that $\int_K\|F(k)\|_{\cH_\pi}^{\,\,2}
d\mu_K(k)<\infty$. As usual, identify functions
that agree $\mu_K$-almost everywhere and equip
$L^2(K,\cH_\pi)$ with the natural inner product.
Then $L^2(K,\cH_\pi)$  is a Hilbert space that is
Hilbert space isomorphic to $\cH_{{\rm ind}\pi}$.
To define the unitary map $W:\cH_{{\rm ind}\pi}\to
L^2(K,\cH_\pi)$, let $\cX$ consist of all
elements of $\cH_{{\rm ind}\pi}$ that are continuous
(more precisely, have a continuous member of the
equivalence class) and let $\cY$ denote the
continuous members of $L^2(K,\cH_\pi)$. Then 
$\cX$ is a dense subspace of $\cH_{{\rm ind}\pi}$ and
$\cY$ is a dense subspace of $L^2(K,\cH_\pi)$.
For $\xi\in \cX$, let $W\xi$ denote the restriction
of $\xi$ to $K$. Then $W\xi\in\cY$. For any $F\in\cY$, define $\xi:G\to \cH_\pi$ by 
$\xi(x)=\pi(h_x^{-1})F(k_x)$, for all $x\in G$. Then
$\xi$ is continuous, $\xi(xh)=\pi(h^{-1})\xi(x)$,
for all $x\in G, h\in H$, and
$\int_{G/H}
\|\xi\big(p(\omega)\big)\|_{\cH_\pi}^{\,\,2}
d\mu_\rho(\omega)=\int_K
\|\xi(k)\|_{\cH_\pi}^{\,\,2}
d\mu_K(k)=\int_K
\|F(k)\|_{\cH_\pi}^{\,\,2}
d\mu_K(k)<\infty$. Thus $\xi\in\cX$ and, clearly,
$W\xi=F$. So $W$ is an
isometry of $\cX$ onto $\cY$. Therefore, $W$
extends to a unitary map, also denoted $W$, of
$\cH_{{\rm ind}\pi}$ onto $L^2(K,\cH_\pi)$.

We can now formulate a proposition that corrects 
and clarifies the expression given in Example 2.29
of \cite{KT} for a representation equivalent to 
${\rm ind}_H^G\pi$ in the special situation we are considering.

\begin{prop}\label{KH_ind}
Let $H$ and $K$ be closed subgroups of a 
locally compact group $G$ that satisfy $K\cap H=
\{e\}$ and $(k,h)\to kh$ is a
homeomorphism of $K\times H$ onto $G$. Let
$\pi$ be a representation of $H$. Then
${\rm ind}_H^G\pi$ is equivalent to $\sigma^\pi$
acting on $L^2(K,\cH_\pi)$ by
\begin{equation}\label{KH_ind_formula}
\sigma^\pi(x)F(k)=\textstyle
\left[\frac{\Delta_H(h_{x^{-1}k})}{\Delta_G(h_{x^{-1}k})}\right]^{1/2}\pi\left(h_{x^{-1}k}^{\,\,\,\,-1}\right)F(x^{-1}\cdot k), \text{ for a.e. } k\in K,
\end{equation}
for all $F\in L^2(K,\cH_\pi)$ and for every $x\in G$.
\end{prop}
\begin{proof}
For each $x\in G$, let 
$\sigma^\pi(x)=W{\rm ind}_H^G(x)W^{-1}$. Then
$\sigma^\pi$ is a representation of $G$ whose Hilbert space is $L^2(K,\cH_\pi)$. To show that $\sigma^\pi$
satisfies \eqref{KH_ind_formula},
it suffices to verify \eqref{KH_ind_formula} for
$F\in\cY$, where $\cY$ is the dense subspace defined
above. Let $F\in\cY$ and let $\xi=W^{-1}F\in\cX$.
For each $x\in G$,
\[
\begin{split}
\sigma^\pi(x)F(k) & =
W\big({\rm ind}_H^G\pi(x)\xi\big)(k)=
{\rm ind}_H^G\pi(x)\xi(k)=\textstyle
\left[\frac{\rho(x^{-1}k)}{\rho(k)}\right]^{1/2}\xi(x^{-1}k)\\
& = \textstyle
\left[\frac{\rho(x^{-1}k)}{\rho(k)}\right]^{1/2}W^{-1}F(x^{-1}k)=
\left[\frac{\rho(x^{-1}k)}{\rho(k)}\right]^{1/2}\pi\left(h_{x^{-1}k}^{\,\,\,\,-1}\right)F(x^{-1}\cdot k),
\end{split}
\]
for all $k\in K$. But 
$\rho(z)=\frac{\Delta_H(h_z)}{\Delta_G(h_z)}$, for
all $z\in G$. Thus 
$\rho(x^{-1}k)=
\frac{\Delta_H(h_{x^{-1}k})}{\Delta_G(h_{x^{-1}k})}$,
while $\rho(k)=1$, for any $k\in K$. Therefore,
\eqref{KH_ind_formula} holds.
\end{proof}

This formula simplifies further when $H$ is 
an abelian normal subgroup and $\pi$ is a one
dimensional representation of $H$. 
When this simplified form is used below, 
$H$ is actually a vector subgroup.
Suppose $n\in \N$ and $K_0$ is a closed subgroup
of $\gln$. Let 
\[
G=\R^n\rtimes K_0=\{[\ux,A]:\ux\in\R^n, A\in K_0\}.
\]
Let $H=\{[\ux,{\rm id}]:\ux\in\R^n\}$. Then 
$H$ is an abelian normal closed subgroup of $G$.
Let $K=\{[\u0,A]:A\in K_0\}$, a closed subgroup
of $G$. We have $K\cap H=\{[\u0,{\rm id}]\}$
and $\big([\u0,A],[\ux,{\rm id}]\big)\to
[\u0,A][\ux,{\rm id}]=[A\ux,A]$ is a homeomorphism
of $K\times H$ with $G$. Note that
\begin{equation}
k_{[\ux,A]}=[\u0,A]\quad\text{and}\quad
h_{[\ux,A]}=[A^{-1}\ux,{\rm id}], \text{ for all }
[\ux,A]\in G.
\end{equation}

The modular function of $G$ is given by
\[
\Delta_G[\ux,A]=\frac{\Delta_{K_0}(A)}{|\det(A)|},
\text{ for all } [\ux,A]\in G.
\]
Note that $\Delta_G\equiv 1$ on $H$ and $H$, itself, 
is unimodular. So $\left[\frac{\Delta_H[\ux,\u0]}{\Delta_G[\ux,\u0]}\right]^{1/2}=1$, for all 
$[\ux,\u0]\in H$.

The irreducible representations of $H$ are 
all of the form $\chi_{\uomega}$, for $\uomega
\in\widehat{\R^n}$, where
\[
\chi_{\uomega}[\ux,{\rm id}]=e^{2\pi i\uomega\ux},
\text{ for all } [\ux,{\rm id}]\in H.
\]
\begin{cor}\label{KH_ind_2}
Let $G=\R^n\rtimes K_0$, where $K_0$ is a closed
subgroup of $\gln$. Let $H=\{[\ux,{\rm id}]:
\ux\in\R^n\}$ and let $\uomega\in\widehat{\R^n}$.
Then ${\rm ind}_H^G\chi_{\uomega}$ is unitarily
equivalent to $\sigma^{\uomega}$, which
acts on $L^2(K_0)$ as follows: For $[\ux,A]\in G$
and $f\in L^2(K_0)$,
\[
\sigma^{\uomega}[\ux,A]f(B)=
e^{2\pi i\uomega B^{-1}\ux}f(A^{-1}B), \text{ for all }
B\in K_0.
\]
\end{cor}
\begin{proof}
For $[\ux,A]\in G$ and $[\u0,B]\in K$,
$[\ux,A]^{-1}[\u0,B]=[-A^{-1}\ux,A^{-1}B]$, so
\[
[\ux,A]^{-1}\cdot[\u0,B]=
k_{[\ux,A]^{-1}[\u0,B]}=[\u0,A^{-1}B]\quad
\text{and}\quad h_{[\ux,A]^{-1}[\u0,B]}=
[-B^{-1}\ux,{\rm id}].
\]
By Proposition \ref{KH_ind}, 
${\rm ind}_H^G\chi_{\uomega}$ is unitarily
equivalent to $\sigma^{\chi_{\uomega}}$ acting on
$L^2(K)$ by, for $[\ux,A]\in G$ and $f\in L^2(K)$,
\[
\sigma^{\chi_{\uomega}}[\ux,A]f[\u0,B]=
\chi_{\uomega}\left(h_{[\ux,A]^{-1}[\u0,B]}^{\,\,\,-1}\right)f[\u0,A^{-1}B]=
e^{2\pi i\uomega B^{-1}\ux}f[\u0,A^{-1}B],
\]
for all $[\u0,B]\in K$. Let $U:L^2(K)\to L^2(K_0)$
be the obvious unitary map $Uf(B)=f[\u0,B]$, for
all $B\in K_0$ and $f\in L^2(K_0)$. Then 
define $\sigma^{\uomega}$ acting on $L^2(K_0)$ by
$\sigma^{\uomega}[\ux,A]=
U\sigma^{\chi_{\uomega}}[\ux,A]U^{-1}$, for all
$[\ux,A]\in G$. Thus,
$\sigma^{\uomega}[\ux,A]f(B)=
e^{2\pi i\uomega B^{-1}\ux}f(A^{-1}B)$, for all 
$B\in K_0$ and 
${\rm ind}_H^G\chi_{\uomega}$ is unitarily
equivalent to both $\sigma^{\chi_{\uomega}}$
and $\sigma^{\uomega
}$.
\end{proof}

\section{The Affine Group on $\R^n$}
\noindent  
In this section, we decompose the left regular
representation of $G_n$ as an infinite multiple
of a representation induced from a character
of the normal vector subgroup we now call $N$, where
$N=\{[\uy,{\rm id}]:\uy\in\R^n\}$. Note that,
for $[\ux,A]\in G_n$ and $[\uy,{\rm id}]\in N$,
\[
[\ux,A]^{-1}[\uy,{\rm id}][\ux,A]=
\left[-A^{-1}\ux,A^{-1}\right][\uy+\ux,A]=
\left[A^{-1}\uy,{\rm id}\right].
\]
For 
$\uomega\in\widehat{\R^n}$, we continue to denote by
$\chi_{\uomega}$ the character of $N$ given by
\[
\chi_{\uomega}[\uy,{\rm id}]=e^{2\pi i\uomega
\uy},\quad\text{for all}\,\,[\uy,{\rm id}]
\in N.
\]
Then $\widehat{N}=\{\chi_{\uomega}:\uomega
\in\widehat{\R^n}\}$.
For $[\ux,A]\in G_n$ and $\chi\in\widehat{N}$,
$[\ux,A]\cdot\chi\in\widehat{N}$ is defined
by
\[
\big([\ux,A]\cdot\chi\big)[\uy,{\rm id}]=
\chi\big([\ux,A]^{-1}[\uy,{\rm id}][\ux,A]
\big)=\chi\left[A^{-1}\uy,{\rm id}\right].
\]
Therefore, the action of $G_n$ on $\widehat{N}$
reduces to, for $[\ux,A]\in G_n$ and 
$\uomega\in\widehat{\R^n}$,
$[\ux,A]\cdot\chi_{\uomega}=\chi_{\uomega
A^{-1}}$. All information about this action is obtained from 
the action of $\gln$ on $\widehat{\R^n}$
given by $(A,\uomega)\to\uomega A^{-1}$.
There are just two orbits, $\{\u0\}$ and
$\cO_n =\widehat{\R^n}\setminus\{\u0\}$. Let
$\uomega_0=(1,0,\cdots,0)$. Then
$\cO_n =\{\uomega_0A^{-1}:A\in\gln\}=
\{\uomega_0A:A\in\gln\}$. It is not difficult
to construct a measurable map $\gamma:\cO_n \to\gln$ that
satisfies  
\begin{equation}\label{gamma_defn}
\uomega_0\gamma(\uomega)^{-1}=\uomega
\,\,\text{(equivalently}\,\,
\uomega\gamma(\uomega)=\uomega_0),\quad
\text{for all}\,\,\uomega\in\cO_n .
\end{equation}
Fix such a map $\gamma$. 
For $\uomega\in\widehat{\R^n}$, by Corollary
\ref{KH_ind_2},
the induced
representation ${\rm ind}_N^{G_n}\chi_{\uomega}$
is unitarily equivalent to a representation we denote
$\pi^{\uomega}$
which acts on $L^2\big(\gln\big)$ as follows:
For $[\ux,A]\in G_n$,
\begin{equation}\label{pi_omega}
\pi^{\uomega}[\ux,A]f(B)=
e^{2\pi i\uomega B^{-1}\ux}f\left(A^{-1}B\right),
\quad\text{for all}\,\,B\in\gln, f\in 
L^2\big(\gln\big).
\end{equation}
It is a basic result of Mackey theory that inducing
two representations from the same orbit results in
equivalent representations (see Proposition 2.39 of
\cite{KT}, for example). Thus, for each $\uomega
\in\cO_n $, we should have  $\pi^{\uomega}\sim
\pi^{\uomega_0}$. There is some value in identifying
the unitary map that institutes this equivalence.
Define $V_{\uomega}:L^2\big(\gln\big)\to
L^2\big(\gln\big)$ by, for 
$f\in L^2\big(\gln\big)$,
\[
V_{\uomega}f(B)=f\big(B\gamma(\uomega)^{-1}
\big),\quad\text{for all}\,\,B\in\gln.
\]
Since $\gln$ is unimodular,
$\|V_{\uomega}f\|_2=\|f\|_2$, for all
$f\in L^2\big(\gln\big)$. It is also clear
that $V_{\uomega}$ is linear, one to one, and
onto. So $V_{\uomega}$ is a unitary map of
$L^2\big(\gln\big)$ with itself. For
$[\ux,A]\in G_n$, $f\in L^2\big(\gln\big)$,
and $B\in\gln$,
\[
V_{\uomega}\pi^{\uomega}[\ux,A]
V_{\uomega}^{-1}f(B)=\pi^{\uomega}[\ux,A]
V_{\uomega}^{-1}f\big(B\gamma(\uomega)^{-1}\big)
=e^{2\pi i\uomega\gamma(\uomega)B^{-1}\ux}
V_{\uomega}^{-1}f\big(A^{-1}B\gamma(\uomega)^{-1}\big)
\qquad\qquad
\]
\[
\qquad\qquad\quad\,
=e^{2\pi i\uomega\gamma(\uomega)B^{-1}\ux}
f\big(A^{-1}B\big)=e^{2\pi i\uomega_0 B^{-1}\ux}
f\big(A^{-1}B\big)=\pi^{\uomega_0}[\ux,A]
f(B).
\]
This shows that $\pi^{\uomega}\sim
\pi^{\uomega_0}$.

Our goal now is to establish an explicit
unitary equivalence of the left regular
representation, $\lambda_{G_n}$, of $G_n$
with an infinite multiple of $\pi^{\uomega_0}$.

For any $f\in L^2(G_n)$, define $Uf$ on $\R^n\times
\gln$ by $Uf(\uy,B)=f[B\uy,B]$,
for all $(\uy,B)\in\R^n\times\gln$. It is 
straightforward to verify that
$U$ is a unitary map of $L^2(G_n)$ onto 
$L^2(\R^n\times\gln)$ when 
$\R^n\times\gln$ is equipped with the
product of Lebesgue measure on $\R^n$ with Haar measure
on $\gln$. Moreover, $U^{-1}:
L^2\big(\R^n\times\gln\big)\to L^2(G_n)$ is
given by
\[
U^{-1}f[\uy,B]=f(B^{-1}\uy,B),\quad
\text{for}\,\,[\uy,B]\in G_n, f\in 
L^2\big(\R^n\times\gln\big).
\]

Let $\cF_1:
L^2\big(\R^n\times\gln\big)\to
L^2\big(\widehat{\R^n}\times\gln\big)$ be
the unitary map such that, for $f\in C_c\big(\R^n\times\gln\big)$
and any $(\uomega,A)\in 
\widehat{\R^n}\times\gln$,
\[
\cF_1f(\uomega,B)=\int_{\R^n}f(\uy,B)
e^{2\pi i\uomega\uy}d\uy.
\]
The left regular representation $\lambda_{G_n}$
of $G_n$ is unitarily equivalent, via 
$\cF_1\circ U$ to a unitary representation
$\widetilde{\lambda_{G_n}}$ acting on
$L^2\big(\widehat{\R^n}\times\gln\big)$. That is
\[
\widetilde{\lambda_{G_n}}[\ux,A]=(\cF_1\circ U)
\lambda_{G_n}[\ux,A](\cF_1\circ U)^{-1},
\text{ for all } [\ux,A]\in G_n.
\]
A short computation shows that, for 
$f\in L^2\big(\widehat{\R^n}\times\gln\big)$,
$[\ux,A]\in G_n$,
and almost every 
$(\uomega,B)\in \widehat{\R^n}\times\gln$,
we have
\[
\widetilde{\lambda_{G_n}}[\ux,A]f(\uomega,B) 
 =e^{2\pi i\uomega B^{-1}\ux}
f\left(\uomega, A^{-1}B\right).
\]
Notice
the similarity with the $\pi^{\uomega}$. 
This shows how we could write the left
regular representation as a direct integral
of the $\pi^{\uomega}$. But 
$\pi^{\uomega}$
is equivalent to $\pi^{\uomega_0}$, for each
$\uomega\in\cO_n $. Thus,
$\lambda_{G_n}$ must be equivalent to a
multiple of $\pi^{\uomega_0}$. We will make 
this equivalence explicit.

Note that $\cO_n $ is a co-null open subset of
$\widehat{\R^n}$. We will consider Lebesgue
measure on $\cO_n $ as its standard measure,
so $L^2(\cO_n \times\gln)$ is the same Hilbert space as
$L^2\big(\widehat{\R^n}\times\gln\big)$.
We define $W$ on $L^2(\cO_n \times\gln)$ as 
follows: For $f\in L^2(\cO_n \times\gln)$, 
\[
(Wf)(\uomega,B)=
f\big(\uomega,B\gamma(\uomega)^{-1}\big),
\text{ for a.e. }
(\uomega,B)\in\cO_n \times\gln.
\]
Then $Wf$ is measurable and, using Fubini's Theorem and that 
$\gln$ is unimodular,
\[
\begin{split}
\int_{\gln}\int_{\cO_n }|(Wf)(\uomega,B)|^2
d\uomega\,d\mu_{\gln}(B) & =
\int_{\cO_n }\int_{\gln}
|f\big(\uomega,B\gamma(\uomega)^{-1}\big)|^2
d\mu_{\gln}(B)\,d\uomega\\
& = \int_{\gln}\int_{\cO_n }
|f\big(\uomega,B\big)|^2
d\uomega\,d\mu_{\gln}(B)=\|f\|_2^{\,\,2}.
\end{split} 
\]
Thus $Wf\in L^2(\cO_n \times\gln)$ and $W$ is
a linear isometry on $L^2(\cO_n \times\gln)$.
Clearly, $W$ is onto and $W^{-1}$ is given by
$\big(W^{-1}g\big)(\uomega,B)=
g\big(\uomega,B\gamma(\uomega)\big)$, for all
$(\uomega,B)\in\cO_n \times\gln$. So $W$ is a unitary.

Define the unitary representation 
$\lambda_{G_n}^0$ of $G_n$ on 
$L^2(\cO_n \times\gln)$ by, for $[\ux,A]\in G_n$,
\[
\lambda_{G_n}^0[\ux,A]=
W\widetilde{\lambda_{G_n}}[\ux,A]W^{-1}.
\]
So, for $f\in L^2\big(\cO_n \times\gln\big)$
and a.e. $(\uomega,B)\in\cO_n \times\gln$,
\[
\lambda_{G_n}^0[\ux,A]f(\uomega,B)=
W\widetilde{\lambda_{G_n}}[\ux,A]W^{-1}
f(\uomega,B)=
\widetilde{\lambda_{G_n}}[\ux,A]W^{-1}
f\big(\uomega,B\gamma(\uomega)^{-1}\big)
\]
\[
=e^{2\pi i\uomega\gamma(\uomega)B^{-1}\ux}
\left(W^{-1}f\right)\big(\uomega,A^{-1}B
\gamma(\uomega)^{-1}\big)=
e^{2\pi i\uomega_0B^{-1}\ux}
f\big(\uomega,A^{-1}B\big).
\]
Therefore, we have $\lambda_{G_n}^0[\ux,A]f(\uomega,B)=
e^{2\pi i\uomega_0B^{-1}\ux}
f\big(\uomega,A^{-1}B\big)$, for a.e.
$(\uomega,B)\in\cO_n \times\gln$, for all
$f\in L^2\big(\cO_n \times\gln\big)$, and all
$[\ux,A]\in G_n$. The representation
$\lambda_{G_n}^0$ is equivalent to 
$\lambda_{G_n}$.

Let $\eta
\in L^2(\cO_n)$ be such that 
$\|\eta\|_{L^2(\cO_n)}=1$.
Let $\cH_\eta=\{\eta\otimes g:g\in 
L^2\big(\gln\big)\}$, where
$\eta\otimes g$ is defined by
$(\eta\otimes g)(\uomega,B)=\eta(\uomega)g(B)$,
for a.e. $(\uomega,B)\in\cO_n \times\gln$. Then
$\cH_\eta$ is a closed subspace of 
$L^2\big(\cO_n \times\gln\big)$. Note that
\begin{equation}\label{lambda_0_equiv_pi_omega}
\lambda_{G_n}^0[\ux,A](\eta\otimes g)(\uomega,B)=
e^{2\pi i\uomega_0B^{-1}\ux}
\eta(\uomega)g\big(A^{-1}B\big)=
\eta(\uomega)\pi^{\uomega_0}[\ux,A]g(B),
\end{equation}
for a.e. $(\uomega,B)\in \cO_n \times\gln$ and all
$g\in L^2\big(\gln\big)$. 

Define 
$W_\eta:L^2\big(\gln\big)
\to L^2\big(\cO_n \times\gln\big)$
by $W_\eta f=\eta\otimes f$, 
for $f\in L^2\big(\gln\big)$.
Clearly, $W_\eta$ is a linear isometry with 
$W_\eta L^2(\cO_n)=\cH_\eta$, for any unit vector
$\eta$ in $L^2(\cO_n)$.
Moreover, \eqref{lambda_0_equiv_pi_omega} shows that
$\cH_\eta$ is $\lambda_{G_n}^0$-invariant and
$W_\eta$ intertwines $\pi^{\uomega_0}$ with
$\lambda_{G_n}^0$ restricted to $\cH_\eta$.
We can now exhibit $\lambda_{G_n}^0$ as a multiple
of $\pi^{\uomega_0}$.

\begin{prop} 
Let $\{\eta_j:j\in J\}$ be an orthonormal 
basis of $L^2(\cO_n )$. For each 
$j\in J$, $\cH_{\eta_j}$ is a closed 
$\lambda_{G_n}^0$-invariant subspace of 
$L^2(\cO_n \times\gln)$
and the map $W_{\eta_j}$ intertwines 
$\pi^{\uomega_0}$ with the restriction of 
$\lambda_{G_n}^0$ to $\cH_{\eta_j}$.
Moreover, 
$L^2(\cO_n \times\gln)
=\textstyle\sum^{\oplus}_{j\in J}\cH_{\eta_j}.
$
\end{prop} 

Returning to a single $\eta\in L^2(\cO_n )$ with
$\|\eta\|_{L^2(\cO_n)}=1$, 
let us see where $\eta\otimes f$
goes as we map it with the above unitaries
back into $L^2(G_n)$.
First, we have 
$W^{-1}(\eta\otimes f)(\uomega,B)=
\eta(\uomega)f\big(B\gamma(\uomega)\big)$,
for $(\uomega,B)\in\cO_n \times\gln$. Next,
\[
\cF_1^{-1}W^{-1}(\eta\otimes f)(\uy,B)=
\int_{\widehat{\R^n}}
\eta(\uomega)f\big(B\gamma(\uomega)\big)
e^{-2\pi i\uomega\uy}d\uomega.
\]
Finally,
\[
\begin{split}
U^{-1}\cF_1^{-1}W^{-1}(\eta\otimes f)[\uy,B] & =
\cF_1^{-1}W^{-1}(\eta\otimes f)(B^{-1}\uy,B)\\
& =
\int_{\widehat{\R^n}}
\eta(\uomega)f\big(B\gamma(\uomega)\big)
e^{-2\pi i\uomega B^{-1}\uy}d\uomega.
\end{split}
\]
We can define $U_\eta:L^2(\gln)\to L^2(G_n)$ by
\begin{equation}\label{U_eta}
U_\eta f[\uy,B]=\int_{\widehat{\R^n}}
\eta(\uomega)f\big(B\gamma(\uomega)\big)
e^{-2\pi i\uomega B^{-1}\uy}d\uomega,
\end{equation}
for a.e. $[\uy,B]\in G_n, 
f\in L^2(\gln)$.
Thus, we obtain the following result.
\begin{prop}\label{Intertwine}\label{U_eta_iso}
Let $\eta\in L^2(\cO_n )$ satisfy 
$\|\eta\|_{L^2(\cO_n)}=1$.
Then $U_\eta$ is an isometric linear map of
$L^2(\gln)$ into $L^2(G_n)$ that intertwines
$\pi^{\uomega_0}$ with $\lambda_{G_n}$.
\end{prop}
Fix an orthonormal basis $\{\eta_j:j\in J\}$
of $L^2(\cO_n )$. Let 
$\cL_{\eta_j}=U^{-1}\cF_1^{-1}W^{-1}\cH_{\eta_j}$,
for each $j\in J$.
Since $U^{-1}\cF_1^{-1}W^{-1}$ is a unitary map
of $L^2(\cO_n \times\gln)$ onto $L^2(G_n)$, we see
that $L^2(G_n)$ is the direct sum of
the $\cL_{\eta_j}$
Thus, we have a decomposition of the left 
regular representation of $G_n$. Note that 
$L^2(\cO_n )$ is identified with $L^2(\widehat{\R^n})$.
\begin{theorem}\label{decomp_GLn_1} 
Let $\cO_n =\widehat{\R^n}\setminus\{\u0\}$, 
let $\uomega_0=
(1,0,\cdots,0)$, and define the representation
$\pi^{\uomega_0}$ of $G_n$ by
$\pi^{\uomega_0}[\ux,A]f(B)=
e^{2\pi i\uomega_0 B^{-1}\ux}f\left(A^{-1}B\right)$,
for $B\in\gln$, $f\in L^2\big(\gln\big)$, and
$[\ux,A]\in G_n$.
Fix a measurable map
$\gamma:\cO_n \to\gln$ such that 
$\uomega_0\gamma(\uomega)^{-1}=\uomega$, for all
$\uomega\in\cO_n $. 
Let $\{\eta_j:j\in J\}$ be an orthonormal basis
of $L^2(\widehat{\R^n})$.  For each 
$j\in J$, define $U_{\eta_j}:L^2(\gln)\to L^2(G_n)$
by
\[
U_{\eta_j} f[\uy,B]=\int_{\widehat{\R^n}}
\eta_j(\uomega)f\big(B\gamma(\uomega)\big)
e^{-2\pi i\uomega B^{-1}\uy}d\uomega,
\]
for $[\uy,B]\in G_n$, and $f\in L^2(\gln)$. Let
$\cL_{\eta_j}=U_{\eta_j}L^2(\gln)$. 
Then $\cL_{\eta_j}$ is a closed 
$\lambda_{G_n}$-invariant subspace of $L^2(G_n)$
and $U_{\eta_j}$ intertwines $\pi^{\uomega_0}$
with the restriction of $\lambda_{G_n}$ to
$\cL_{\eta_j}$.
Moreover,
$L^2(G_n)=\sum_{j\in J}^{\oplus}\cL_{\eta_j}$.
\end{theorem}

When $n=1$, ${\rm GL}_1(\R)$ can be identified
with $\R^*$ and
$G_1$ identified with $\R\rtimes\R^*$. We recall that
$\int_{\R^*}f\,d\mu_{\R^*}=\int_{\R}f(b)
\frac{db}{|b|}$,
where the integral on the right hand side is
the Lebesgue integral on $\R$, and
$\int_{G_1}f\,d\mu_{G_1}=\int_{\R}\int_{\R}
f[y,b]\frac{dy\,db}{b^2}$.
We continue to write $N=\{[y,1]:y\in\R\}$
and $\widehat{N}=\{\chi_\omega:\omega\in\R\}$,
where $\chi_\omega[y,1]=e^{2\pi i\omega y}$,
for $[y,1]\in N$. Select 
$\uomega_0=\omega_0=1$. Then, from \eqref{pi_omega},
$\pi^{\uomega_0}=\pi^1$ acts on $L^2(\R^*)$
via, for $[x,a]\in G_1$, $f\in L^2(\R^*)$, and a.e.
$b\in\R^*$,
\begin{equation}\label{pi_1}
\pi^1[x,a]f(b)=
e^{2\pi ib^{-1}x}f\big(a^{-1}b\big).
\end{equation}
Now, the $\R^*$-orbit $\cO_1=\widehat{\R}\setminus
\{0\}$ is a free orbit. Condition \eqref{gamma_defn}
forces $\gamma:\cO_1
\to\R^*$ to actually be the homeomorphism given by
$\gamma(\omega)=\omega^{-1}$. We can use $\gamma$ to 
move $\pi^1$ to a representation acting on
$L^2(\widehat{\R})$. Note that
$L^2(\widehat{\R})=L^2(\cO_1)$ when $\cO_1$ is
equipped with Lebesgue measure. For $f\in L^2(\R^*)$,
define $Wf$ as a Borel function on $\widehat{\R}$
such that
\[
(Wf)(\omega)=|\omega|^{-1/2}f\big(\gamma(\omega)\big), \text{ for a.e. }\omega\in\cO_1.
\]
Then 
$\int_{\widehat{\R}}|(Wf)(\omega)|^2d\omega=
\int_{\widehat{\R}}|f(\omega^{-1})|^2|\omega|^{-1}\,
d\omega=\int_{\R}|f(u)|^2\frac{du}{|u|}=
\|f\|_{L^2(\R^*)}^{\,\,\,2}$. Thus 
$Wf\in L^2(\widehat{\R})$. It is clear that
$W$ is a unitary map of $L^2(\R^*)$ onto
$L^2(\widehat{\R})$.
Let 
$\pi[x,a]=W\pi^1[x,a]W^{-1}$,
for all $[x,a]\in G_1$. For any $\xi\in L^2(\widehat{\R})$, we have (writing 
$f=W^{-1}\xi$)
\[
\pi[x,a]\xi(\omega)=W\pi^1[x,a]f(\omega)
=|\omega|^{-1/2}\pi^1[x,a]f\big(\omega^{-1}\big)
=|\omega|^{-1/2}e^{2\pi i\omega x}
f\big(a^{-1}\omega^{-1}\big)
\]
\[
=|a|^{1/2}
e^{2\pi i\omega x}\big|\omega a\big|^{-1/2}
f\big((\omega a)^{-1}\big)=
|a|^{1/2}
e^{2\pi i\omega x}\xi(\omega a),\,\,\text{for 
a.e. }\omega\in\widehat{\R}.
\]
Thus, $\pi[x,a]\xi(\omega)=|a|^{1/2}
e^{2\pi i\omega x}\xi(\omega a)$, for a.e.
$\omega\in\widehat{\R},\xi\in L^2(\widehat{\R})$, and $[x,a]\in G_1$. Finally, we use the
inverse Fourier transform to move $\pi$ to
a representation on $L^2(\R)$. Let
$\rho[x,a]=\cF^{-1}\pi[x,a]\cF$, for all
$[x,a]\in G_1$. Then, for $f\in L^2(\R)$,
$\xi=\cF f$, 
and a.e. $t\in\R$, 
\[
\begin{split}
\rho[x,a]f(t) & = \cF^{-1}\pi[x,a]\xi(t)=
\int_{\widehat{\R}}\pi[x,a]\xi(\omega)
e^{-2\pi i\omega t}d\omega\\
& =\int_{\widehat{\R}}|a|^{1/2}
e^{2\pi i\omega x}\xi(\omega a)
e^{-2\pi i\omega t}d\omega
=\int_{\widehat{\R}}|a|^{1/2}
\xi(\omega a)
e^{-2\pi i\omega (t-x)}d\omega\\
& =\int_{\widehat{\R}}|a|^{-1/2}
\xi(\nu) 
e^{-2\pi i\nu a^{-1}(t-x)}d\nu
=|a|^{-1/2}f\big(a^{-1}(t-x)\big).
\end{split}
\]
Thus, $\rho$ is just the natural representation of
$G_1$ on $L^2(\R)$, compare with \eqref{natural},
which is square-integrable by Proposition 
\ref{open_orbit_sq_int}.
We now have a corollary of Theorem \ref{decomp_GLn_1}.
\begin{cor}\label{mult_G_1}
The left regular representation $\lambda_{G_1}$ is
equivalent to $\aleph_0\cdot\rho$, where $\rho$ is
the natural representation of $G_1$ on $L^2(\R)$
and $\rho$ is square-integrable.
\end{cor}
Since $\pi^1$ is equivalent to $\rho$, we have
that $\pi^1$ is square-integrable. For later
use, we need to identify which $g\in L^2(\R^*)$ is
such that $V_g$, defined via the representation
$\pi^1$, is an isometry of $L^2(\R^*)$ into
$L^2(G_2)$.
\begin{prop}\label{C_pi_1}
The Duflo-Moore operator for $\pi^1$
is given by $C_{\pi^1}g(t)=|t|^{1/2}g(t)$, a.e. 
$t\in\R^*$, for $g\in\cD_{\pi^1}$, where
$\cD_{\pi^1}=\left\{h\in L^2(\R^*):
\int_{\R^*}\big||t|^{1/2}h(t)\big|^2d\mu_{\R^*}(t)
<\infty\right\}$.
\end{prop}
\begin{proof}
Calculations modeled on the standard proof, such as
used in Example 2.28 of \cite{Fuh}, show that, 
for any $f,g\in L^2(\R^*)$, 
\begin{equation}\label{V_fg}
\int_{G_1}|V_gf[x,a]|^2
d\mu_{G_1}\big([x,a]\big)=
\|f\|_2^{\,\,2}
\int_{\R^*}\big||t|^{1/2}
g(t)\big|^2d\mu_{\R^*}(t).
\end{equation}
If $f\in L^2(\R^*)$ is nonzero, then the
right hand side of \eqref{V_fg} is finite if and
only if 
\[
\int_{\R^*}\big||t|^{1/2}
g(t)\big|^2d\mu_{\R^*}(t)<\infty.
\]
An appeal to Theorem \ref{Du_Mo} completes the
proof.
\end{proof}
\begin{cor}\label{pi_1_V_g_isometry}
If $g\in L^2(\R^*)$ satisfies 
$\int_{\R^*}\big||t|^{1/2}
g(t)\big|^2d\mu_{\R^*}(t)=1$, then $V_g$ is
an isometry of $L^2(\R^*)$ into $L^2(G_1)$ that
intertwines $\pi^1$ with $\lambda_{G_1}$.
\end{cor}
Now, we can use basic facts of the theory of induced representations to
show that $\lambda_{G_n}$ is a 
countably infinite multiple of a single
square-integrable representation, for any $n\in \N$.
The proof of the following theorem follows the
ideas in \cite{BT}.
\begin{theorem}\label{mult_G_n}
Let $n\in\N$, there exists a square-integrable 
representation $\sigma_n$ of $G_n$ such that
$\lambda_{G_n}$ is equivalent to 
$\aleph_0\cdot\sigma_n$.
\end{theorem}
\begin{proof}
Use mathematical induction. The case of $n=1$ is
simply Corollary \ref{mult_G_1}. Suppose that
$n\in\N$ and there exists a square-integrable
representation $\sigma_n$ of $G_n$ such that
$\lambda_{G_n}$ is equivalent to 
$\aleph_0\cdot\sigma_n$. Now consider the action
of ${\rm GL}_{n+1}(\R)$ on $\widehat{\R^{n+1}}$.
Let $\uomega_0=(1,0,\cdots,0)\in\widehat{\R^{n+1}}$.
The stability subgroup of $\uomega_0$ is
\[
H_{\uomega_0}=\left\{\begin{pmatrix}
1 & \u0\\
\ux & A
\end{pmatrix}:\ux\in\R^n, A\in\gln\right\},
\]
where $\u0$ here denotes a row of $n$ zeros. Note
that $[\ux,A]\to\begin{pmatrix}
1 & \u0\\
\ux & A
\end{pmatrix}$ is a topological group isomorphism
of $G_n$ with $H_{\uomega_0}$. Let 
$\sigma_n'\begin{pmatrix}
1 & \u0\\
\ux & A
\end{pmatrix}=\sigma_n[\ux,A]$, for each
$\begin{pmatrix}
1 & \u0\\
\ux & A
\end{pmatrix}\in H_{\uomega_0}$. Then 
$\sigma_n'$ is a square-integrable representation
of $H_{\uomega_0}$ and $\lambda_{H_{\uomega_0}}$
is equivalent to $\aleph_0\cdot\sigma_n'$.
Use $N$ to denote
$\{[\uz,{\rm id}]:\uz\in\R^{n+1}\}$, where 
${\rm id}$ is the identity in ${\rm GL}_{n+1}(\R)$.

By Theorem \ref{decomp_GLn_1}, $\lambda_{G_{n+1}}$ is
equivalent to $\aleph_0\cdot\pi^{\uomega_0}$. 
Moreover, $\pi^{\uomega_0}$ is equivalent to
${\rm ind}_N^{G_{n+1}}\chi_{\uomega_0}$. Let
$H=\R^{n+1}\rtimes H_{\uomega_0}=\{[\uz,B]:\uz\in\R^{n+1},B\in H_{\uomega_0}\}$. By the induction in
stages theorem (Theorem 2.47 in \cite{KT})
${\rm ind}_N^{G_{n+1}}\chi_{\uomega_0}$ is 
equivalent to 
${\rm ind}_H^{G_{n+1}}\big({\rm ind}_N^{H}\chi_{\uomega_0}\big)$. Since $\uomega_0 B^{-1}
=\uomega_0$, for any $B\in H_{\uomega_0}$, 
Corollary \ref{KH_ind_2} implies 
${\rm ind}_N^{H}\chi_{\uomega_0}$ is equivalent to
the representation $\sigma^{\uomega_0}$ acting on
$L^2(H_{\uomega_0})$ by, for $[\uz,B]\in H$,
 $\sigma^{\uomega_0}
 [\uz,B]f(C)=
 e^{2\pi i\uomega_0\uz}
 f(B^{-1}C)$, for $C\in H_{\uomega_0}$ and 
 $f\in L^2(H_{\uomega_0})$. That is, 
 $\sigma^{\uomega_0}=\chi_{\uomega_0}\otimes
 (\lambda_{H_{\uomega_0}}\circ q)$, where 
 $q:H\to H_{\uomega_0}$ is the homomorphism given
 by $q[\uz,B]=B$, for all $[\uz,B]\in H$ (see
 Remark 1.45 of \cite{KT}). By the inductive
 hypothesis, $\sigma^{\uomega_0}$, and hence also
 ${\rm ind}_N^{H}\chi_{\uomega_0}$, is equivalent
 to $\chi_{\uomega_0}\otimes\big(\aleph_0\cdot
 (\sigma_n'\circ q)\big)=\aleph_0\cdot\big(
 \chi_{\uomega_o}\otimes(\sigma_n'\circ q)\big)$. The
 process of inducing commutes with taking direct
sums (see Proposition 2.42 of \cite{KT}, for example). Thus,
\[
\pi^{\uomega_0}\sim{\rm ind}_N^{G_{n+1}}\chi_{\uomega_0}\sim
{\rm ind}_H^{G_{n+1}}\big({\rm ind}_N^{H}\chi_{\uomega_0}\big)\sim
\aleph_0\cdot{\rm ind}_H^{G_{n+1}}\big(
 \chi_{\uomega_o}\otimes(\sigma_n'\circ q)\big).
\]
Let $\sigma_{n+1}= {\rm ind}_H^{G_{n+1}}\big(
 \chi_{\uomega_o}\otimes(\sigma_n'\circ q)\big)$.
 By Theorem 4.24 of \cite{KT} or Theorem 6.42 in
 \cite{Fol}, $\sigma_{n+1}$ is an irreducible
 representation of $G_{n+1}$. We have
 $\pi^{\uomega_0}$ is equivalent to $\aleph_0\cdot 
 \sigma_{n+1}$ and $\lambda_{G_{n+1}}$ is 
 equivalent to $\aleph_0\cdot\pi^{\uomega_0}$.
 Therefore, $\lambda_{G_{n+1}}$ is 
 equivalent to $\aleph_0\cdot\sigma_{n+1}$. 
 Finally, since $\sigma_{n+1}$ is irreducible and
 equivalent to a subrepresentation of 
 $\lambda_{G_{n+1}}$, we have that $\sigma_{n+1}$
 is a square-integrable representation of $G_{n+1}$
 (see Theorem 2.25 of \cite{Fuh}, for example).
\end{proof} 

\begin{rem}
Theorem \ref{mult_G_n} says that $G_n$ is an 
{\rm [AR]}-group and $\widehat{G_n^r}$ is a singleton. We have detailed knowledge in the case of 
$n=1$ and $\sigma_1$ could be taken to be either
$\pi^1$ as in \eqref{pi_1} or the natural
representation $\rho$ acting on $L^2(\R)$, or any
other equivalent realization. When $n=2$, we
will reserve the notation $\sigma_2$ for one 
concrete realization of the unique member of
$\widehat{G_2^r}$.
\end{rem}

\section{Factoring $\gl2$ and $G_2$}
For the rest of this paper, the focus is on $G_2$
and obtaining an explicit description of the
square-integrable representation whose existence
is guaranteed by Theorem \ref{mult_G_n}. The first
step is to establish decompositions of $\gl2$
and $G_2$ so that Proposition \ref{KH_ind} can
be applied. Also, in this section, we introduce
a particular choice of the map 
$\gamma$ from the open orbit into $\gl2$
that is a cross-section of the group action.
Properties of this $\gamma$ turn out to be effective 
in simplifying computations.

As usual $N$ denotes the abelian normal subgroup
$\{[\uy,{\rm id}]:\uy
\in\R^2\}$. 
The nontrivial orbit in $\widehat{\R^2}$ will be
denoted $\cO$, rather than $\cO_2$. So
$\cO=\widehat{\R^2}\setminus\{\u0\}$.
 The point $\uomega_0=(1,0)$
will serve as a representative point in $\cO$. The
stability subgroup inside $\gl2$ 
for this point is now denoted $H_{(1,0)}$. So
\[
H_{(1,0)}=\big\{A\in\gl2:(1,0)A=(1,0)\big\}=
\left\{\begin{pmatrix}
1 & 0\\
u & v
\end{pmatrix}:u,v\in\R, v\neq 0
\right\}.
\]
Let $K_0=\left\{\begin{pmatrix}
s & -t\\
t & s
\end{pmatrix}:s,t\in\R,s^2+t^2>0\right\}$. Then
$K_0$ is a closed subgroup of $\gl2$ such that
$K_0\cap H_{(1,0)}=\{{\rm id}\}$. The fact that
$\gl2=K_0H_{(1,0)}$ is established in the following
proposition which is proven by direct computation.

\begin{prop}\label{factor_gl2}
If $A=\begin{pmatrix}
a & b\\
c & d
\end{pmatrix}\in\gl2$, then $A$ can be uniquely
decomposed as the product
$A=M_AC_A$ with $M_A\in K_0$ and
$C_A\in H_{(1,0)}$. In fact 
\[
M_A=\begin{pmatrix}
s & -t\\
t & s
\end{pmatrix},
\text{ with } s=\frac{d(ad-bc)}{b^{2}+d^{2}},
 t=\frac{-b(ad-bc)}{b^{2}+d^{2}},
\]
and
\[
C_A=\begin{pmatrix}
1 & 0\\
u & v
\end{pmatrix},
\text{ with } u=\frac{cd+ab}{(ad-bc)}, 
v=\frac{b^{2}+d^{2}}{(ad-bc)}.
\]
\end{prop}
This factorization leads to a parallel factorization
of $G_2$. Let 
\[
H=\R^2\rtimes H_{(1,0)}=\{[\ux,C]:C\in H_{(1,0)}\}
\text{ and }
K=\{[\u0,M]:M\in K_0\}.
\]
Then $K\cap H=\{[\u0,{\rm id}]\}$ and 
$G_2=KH$. For any $[\ux,A]\in G_2$,
\begin{equation}\label{factor_G_2}
[\ux,A]=[\u0,M_A][M_A^{-1}\ux,C_A],
\end{equation}
where $M_A$ and $C_A$ are as defined in 
Proposition \ref{factor_gl2}. Note that the 
map $\big([\u0,M],[\ux,C]\big)\to[M\ux,MC]$ is
a homeomorphism of $K\times H$ with $G_2$, so
the conditions necessary for Proposition 
\ref{KH_ind} are satisfied when we wish to
induce a representation of $H$ to $G_2$.

 For each $\uomega=
(\omega_1,\omega_2)\in\cO$, there
is a unique element $\gamma(\uomega)\in K_0$ such
that $\gamma(\uomega)\cdot(1,0)=\uomega$. Since
$\uomega=\gamma(\uomega)\cdot(1,0)=
(1,0)\gamma(\uomega)^{-1}$, the top row of
$\gamma(\uomega)^{-1}$ must be 
$(\omega_1 \quad\omega_2)$. Thus 
\[
\gamma(\uomega)^{-1}=\begin{pmatrix}
\omega_1 & \omega_2\\
-\omega_2 & \omega_1
\end{pmatrix}\quad\text{and}\quad
\gamma(\uomega)=\frac{1}{\|\uomega\|^2}
\begin{pmatrix}
\omega_1 & -\omega_2\\ 
\omega_2 & \omega_1
\end{pmatrix}.
\]
We will frequently use that 
$(1,0)\gamma(\uomega)^{-1}
=\uomega$ and $\uomega\gamma(\uomega)=(1,0)$, for
any $\uomega\in\cO$. We also will need the
observation that Haar integration on $K$ is
given by, for $f\in C_c(K)$,
\begin{equation}\label{Haar_K}
\int_K f\,d\mu_K=\int_{\widehat{\R^2}}
f[\u0,\gamma(\uomega)]\,
\frac{d\uomega}{\|\uomega\|^2}=
\int_{\widehat{\R^2}}
f[\u0,\gamma(\uomega)^{-1}]\,
\frac{d\uomega}{\|\uomega\|^2}.
\end{equation}

In our calculations later, various 
matrices related to $A\in\gl2$ and $\uomega\in\cO$
arise and there are a number of identities involving
these matrices that are useful. We also use the
entries of the matrix 
$C_{A^{-1}\gamma(\uomega)}^{\,\,-1}$ and need a
notation for these entries. Let 
$u_{\uomega,A}= 
(0,1)C_{A^{-1}\gamma(\uomega)}^{\,\,-1}\begin{pmatrix}
1\\
0
\end{pmatrix}$ and $v_{\uomega,A}=
(0,1)C_{A^{-1}\gamma(\uomega)}^{\,\,-1}\begin{pmatrix}
0\\
1
\end{pmatrix}$. Then 
$C_{A^{-1}\gamma(\uomega)}^{\,\,-1}=\begin{pmatrix}
1 & 0\\
u_{\uomega,A} & v_{\uomega,A}
\end{pmatrix}$. The identities we need are collected
into a proposition.
\begin{prop}\label{identities} 
Let $A,B\in\gl2$ and $\uomega\in\cO$. Then
\begin{enumerate}[{\rm (a)}]
\item $M_A=\gamma\big((1,0)A^{-1}\big)$ and 
$C_A=\gamma\big((1,0)A^{-1}\big)^{-1}A$,
\medskip
\item $M_{A\gamma(\uomega)}=
\gamma\big(\uomega A^{-1}\big)$ and 
$C_{A\gamma(\uomega)}=
\gamma\big(\uomega A^{-1}\big)^{-1}A
\gamma(\uomega)$,
\medskip
\item $M_{A^{-1}\gamma(\uomega)}=
\gamma\big(\uomega A\big)$ and
$C_{A^{-1}\gamma(\uomega)}=
\gamma\big(\uomega A\big)^{-1}A^{-1}\gamma(\uomega)$,
\medskip
\item $u_{\uomega,AB}=
u_{\uomega,A}+v_{\uomega,A}u_{\uomega A,B}$ and
$v_{\uomega,AB}=v_{\uomega,A}v_{\uomega A,B}$,
\medskip
\item $\det\left(\gamma(\uomega)^{-1}\right)=\|\uomega\|^2$
and 
$\det\left(\gamma(\uomega)\right)=\|\uomega\|^{-2}$,
\medskip
\item $\det\left(C_{A^{-1}\gamma(\uomega)}\right)=
\frac{\|\uomega A\|^2}{\det(A)\|\uomega\|^2}$,
and 
\medskip
\item $v_{\uomega,A}=
\det\left(C_{A^{-1}\gamma(\uomega)}^{\,\,-1}\right)=
\frac{\det(A)\|\uomega\|^2}{\|\uomega A\|^2}$.
\end{enumerate} 
\end{prop}
\begin{proof}
(a) and (c) follow from (b), since $\gamma\big((1,0)\big)$
is the identity matrix. For any $\uomega\in\cO$,
$\gamma\big(\uomega A^{-1}\big)\in K_0$ by
definition of $\gamma$. On the other hand,
\[
\begin{split}
(1,0)\left(\gamma\big(\uomega A^{-1}\big)^{-1}
A\gamma(\uomega)\right) & =
\left((1,0)\gamma\big(\uomega A^{-1}\big)^{-1}
\right)A\gamma(\uomega)\\
& =
\uomega A^{-1}A\gamma(\uomega)=\uomega\gamma(\uomega)
=(1,0).
\end{split}
\]
Thus $C_{A\gamma(\uomega)}=
\gamma\big(\uomega A^{-1}\big)^{-1}A
\gamma(\uomega)$ and $M_{A\gamma(\uomega)}=
\gamma\big(\uomega A^{-1}\big)$, by uniqueness. Clearly (e) is true while (f) and (g) follow from 
(c) and (e). It remains to verify (d). 
By (c), $C_{A^{-1}\gamma(\uomega)}^{\,\,-1}=
\gamma(\uomega)^{-1}A\gamma(\uomega A)$. Thus
\[
\begin{split}
C_{(AB)^{-1}\gamma(\uomega)}^{\,\,-1} & =
\gamma(\uomega)^{-1}(AB)\gamma(\uomega AB)=
\gamma(\uomega)^{-1}A\gamma(\uomega A)
\gamma(\uomega A)^{-1}B\gamma(\uomega AB)\\
& = C_{A^{-1}\gamma(\uomega)}^{\,\,-1}
C_{B^{-1}\gamma(\uomega A)}^{\,\,-1}.
\end{split}
\]
That is,
\[
\begin{split}
\begin{pmatrix}
1 & 0\\
u_{\uomega,AB} & v_{\uomega,AB}
\end{pmatrix} & = \begin{pmatrix}
1 & 0\\
u_{\uomega,A} & v_{\uomega,A}
\end{pmatrix}\begin{pmatrix}
1 & 0\\
u_{\uomega A,B} & v_{\uomega A,B}
\end{pmatrix}\\
& =
\begin{pmatrix}
1 & 0\\
u_{\uomega,A}+v_{\uomega,A}u_{\uomega A,B} & v_{\uomega,A}v_{\uomega A,B}
\end{pmatrix},
\end{split}
\]
which establishes (d).
\end{proof}
The detailed values of $u_{\uomega,A}$ and
$v_{\uomega,A}$ are usually not needed, but
may sometimes be useful.
\begin{prop}\label{u_omegaA_v_omegaA}
Let $A=\begin{pmatrix}
a & b\\
c& d
\end{pmatrix}\in\gl2$ and 
$\uomega=(\omega_1,\omega_2)\in\cO$. Then 
\[
u_{\uomega,A}=\frac{(ac+bd)(\omega_1^2-\omega_2^2)-(a^2+b^2-c^2-d^2)\omega_1\omega_2}{(a\omega_1+c\omega_2)^2+(b\omega_1+d\omega_2)^2}
\]
and
\[
v_{\uomega,A}=\frac{(ad-bc)(\omega_1^2+\omega_2^2)}{(a\omega_1+c\omega_2)^2+(b\omega_1+d\omega_2)^2}.
\]
\end{prop}
\begin{proof}
These are both obtained by straightforward calculation
from the definitions of $u_{\uomega,A}$ and
$v_{\uomega,A}$.
\end{proof}

The map $[u,v]\to\begin{pmatrix}
1 & 0\\
u & v
\end{pmatrix}$ is an isomorphism of the group
$G_1=\R\rtimes\R^*$ with 
 $H_{(1,0)}$.
We saw that $G_1$ has, up to equivalence, one
square-integrable representation. One realization is
$\pi^1$, which acts on $L^2(\R^*)$ as in 
\eqref{pi_1}.
We will simplify notation by  also
considering $\pi^1$ as an irreducible representation
of $H_{(1,0)}$. That is, if $C=\begin{pmatrix}
1 & 0\\
u & v
\end{pmatrix}\in H_{(1,0)}$, 
then we let $\pi^1(C)=\pi^1[u,v]$. Lift $\pi^1$
to $H$. That is, let
$\tilde{\pi}^1[\ux,C]=\pi^1(C)$, for every
$[\ux,C]\in H$.
The representation
$\chi_{(1,0)}\otimes\tilde{\pi}^1$ of $H$, given by
\[
\left(\chi_{(1,0)}\otimes\tilde{\pi}^1\right)[\ux,C]=
\chi_{(1,0)}(\ux)\pi^1(C),\quad\text{for }
[\ux,C]\in H,
\]
is an irreducible representation of $H$ on
$L^2(\R^*)$. Then 
${\rm ind}_H^{G_2}\left(\chi_{(1,0)}\otimes\tilde{\pi}^1\right)$
is an irreducible representation of $G_2$ 
by Theorem 4.24 of \cite{KT}. By Proposition 
\ref{KH_ind},
${\rm ind}_H^G\left(\chi_{(1,0)}\otimes
\tilde{\pi}^1\right)$
is equivalent to a representation $\sigma$ acting on
$L^2\big(K,L^2(\R^*)\big)$. In preparation for
defining $\sigma$, take $[\ux,A]\in G_2$, and 
$[\u0,L]\in K$ and compute $[\ux,A]^{-1}[\u0,L]=
[-A^{-1}\ux,A^{-1}L]$. Now factor
\[
[-A^{-1}\ux,A^{-1}L]=[\u0,M_{A^{-1}L}]
[-M_{A^{-1}L}^{\,\,-1}A^{-1}\ux ,C_{A^{-1}L}],
\] with
the elements
$[\u0,M_{A^{-1}L}]\in K$ and 
$[-M_{A^{-1}L}^{\,\,-1}A^{-1}\ux ,C_{A^{-1}L}]\in H$.
Observe that
\[
C_{_{A^{-1}L}}^{\,\,\,\,-1}M_{A^{-1}L}^{\,\,-1}A^{-1}=
\big(M_{A^{-1}L}C_{_{A^{-1}L}}\big)^{-1}A^{-1}=L^{-1}.
\]
Therefore, 
\[
[-M_{A^{-1}L}^{\,\,-1}A^{-1}\ux ,C_{A^{-1}L}]^{-1}=
[L^{-1}\ux,
C_{_{A^{-1}L}}^{\,\,\,\,-1}].
\]
Thus, Proposition \ref{KH_ind} gives, for $F\in L^2\big(K,L^2(\R^*)\big)$, $[\ux,A]\in G$, and 
$[\u0,L]\in K$, 
\begin{equation}\label{sigma}  
\begin{split}
\sigma[\ux,A]F[\u0,L] & = 
\left|\det\left(C_{_{A^{-1}L}}\right)
\right|^{-1/2}\big(\chi_{(1,0)}\otimes
\pi^1\big)\left[L^{-1}\ux,
C_{_{A^{-1}L}}^{\,\,\,\,-1}\right]
F[\u0,M_{A^{-1}L}]\\
& =\left|\det\left(C_{_{A^{-1}L}}\right)
\right|^{-1/2}e^{2\pi i(1,0)L^{-1}\ux}\,
\pi^1\!\left(C_{_{A^{-1}L}}^{\,\,\,\,-1}\right)
F[\u0,M_{A^{-1}L}].
\end{split}
\end{equation}
Use of
the homeomorphism $\gamma:\cO\to K_0$ and the 
identities collected in Proposition \ref{identities} help make
the expression given in \eqref{sigma} easier to
read. For $F\in L^2\big(K,L^2(\R^*)\big)$, 
$[\ux,A]\in G_2$, and a.e. $\uomega\in\cO$,
\begin{equation}\label{sigma2}
\sigma[\ux,A]F[\u0,\gamma(\uomega)] = 
\textstyle\frac{|\det(A)|^{1/2}\|\uomega\|}{\|\uomega A\|}
e^{2\pi i\uomega\ux}\,
\pi^1\big(\gamma(\uomega)^{-1}A\gamma(\uomega A)\big)
F[\u0,\gamma(\uomega A)].
\end{equation}
In \eqref{sigma2}, $\sigma[\ux,A]F[\u0,\gamma(\uomega)]\in L^2(\R^*)$. Before evaluating it,
we note that 
\[
\gamma(\uomega)^{-1}A\gamma(\uomega A)
=C_{A^{-1}\gamma(\uomega)}^{\,\,-1}=
\begin{pmatrix}
1 & 0\\
u_{\uomega,A} & v_{\uomega,A},
\end{pmatrix}
\]
so $\pi^1\big(\gamma(\uomega)^{-1}A\gamma(\uomega A)\big)=\pi^1[u_{\uomega,A}, v_{\uomega,A}]$, when
we consider $\pi^1$ as a representation of $G_1$.
Using \eqref{pi_1} and \eqref{sigma2}, we have, for
a.e. $t\in\R^*$, 
\begin{equation}\label{sigma_pointwise_formula}
\big(\sigma[\ux,A]F[\u0,\gamma(\uomega)]\big)(t) = 
\textstyle\frac{|\det(A)|^{1/2}\|\uomega\|}{\|\uomega A\|}
e^{2\pi i(\uomega\ux+t^{-1}u_{\uomega,A})}
\big(F[\u0,\gamma(\uomega A)]\big)
(v_{\uomega,A}^{-1}t),
\end{equation}
for a.e. $\uomega\in\cO$, with 
$F\in L^2\big(K,L^2(\R^*)\big)$
and $[\ux,A]\in G_2$. 
By construction, $\sigma$ is equivalent to 
${\rm ind}_H^G\left(\chi_{(1,0)}\otimes
\tilde{\pi}^1\right)$, which is irreducible. 
\begin{prop}
The representation $\sigma$ of $G_2$ given by
\eqref{sigma_pointwise_formula} is irreducible.
\end{prop}

Corollary \ref{pi_1_V_g_isometry}  
says that, if $g\in L^2(\R^*)$
satisfies $\int_{\R^*}\big||\nu|^{1/2}g(\nu)\big|^2
d\mu_{\R^*}(\nu)=1$, then $V_g:L^2(\R^*)\to 
L^2(H_{(1,0)})$ is an isometry that intertwines
$\pi^1$ with $\lambda_{H_{(1,0)}}$. Recall that
\[
V_gf(D)=\langle f,\pi^1(D)g\rangle_{_{L^2(\R^*)}},
\text{ for all } D\in H_{(1,0)}, f\in L^2(\R^*).
\]
Let $\cK_g=V_g L^2(\R^*)$, a closed 
$\lambda_{H_{(1,0)}}$-invariant subspace of 
$L^2(H_{(1,0)})$.

Let $V_g':L^2\big(K,L^2(\R^*)\big)\to 
L^2\big(K,L^2(H_{(1,0)})\big)$ be given by
$\big(V_g'F\big)[\u0,L]=V_g\big(F[\u0,L]\big)$,
for all $[\u0,L]\in K$ and $F\in 
L^2\big(K,L^2(\R^*)\big)$. Since $V_g$ is an
isometry, so is $V_g'$ and the range of $V_g'$
is $L^2(K,\cK_g)$. For $[\ux,A]\in G_2$, $F\in 
L^2\big(K,L^2(\R^*)\big)$, and $\uomega\in\cO$,
\[
\begin{split}
(V_g'\sigma[\ux,A]F)[\u0,\gamma(\uomega)] & =
\textstyle\frac{|\det(A)|^{1/2}\|\uomega\|}{\|\uomega A\|}
e^{2\pi i\uomega\ux}\,V_g\big(
\pi^1\big(\gamma(\uomega)^{-1}A\gamma(\uomega A)\big)
F[\u0,\gamma(\uomega A)]\big)\\
& = 
\textstyle\frac{|\det(A)|^{1/2}\|\uomega\|}{\|\uomega A\|}
e^{2\pi i\uomega\ux}\,
\lambda_{H_{(1,0)}}\big(\gamma(\uomega)^{-1}A\gamma(\uomega A)\big)
V_g\big(
F[\u0,\gamma(\uomega A)]\big).
\end{split}
\]
Thus, $\sigma$ is equivalent to a representation 
$\tilde{\sigma}$ acting on $L^2(K,\cK_g)$ as follows: For
$[\ux,A]\in G_2$,  $\varphi\in 
L^2(K,\cK_g)$, and $\uomega\in\cO$,
\begin{equation}\label{sigma_2}
\big(\tilde{\sigma}[\ux,A]\varphi\big)[\u0,\gamma(\uomega)]=
\textstyle\frac{|\det(A)|^{1/2}\|\uomega\|}{\|\uomega A\|}
e^{2\pi i\uomega\ux}\,
\lambda_{H_{(1,0)}}\big(\gamma(\uomega)^{-1}A\gamma(\uomega A)\big)
\varphi[\u0,\gamma(\uomega A)].
\end{equation}
The next step is to map $L^2(K,\cK_g)$ isometrically
into $L^2\big(\gl2\big)$. 

Let $W_1:L^2\big(\gl2\big)\to 
L^2\big(K,L^2(H_{(1,0)})\big)$ be given by
\[
\big(W_1f[\u0,M]\big)(C)=|\det(C)|^{1/2}f(MC),
\]
for all $C\in H_{(1,0)}$, $[\u0,M]\in K$.
\begin{prop}\label{W_inverse}
The map $W_1$ is a unitary map onto $L^2\big(K,L^2(H_{(1,0)})\big)$
and its inverse is given by, for $F\in 
L^2\big(K,L^2(H_{(1,0)})\big)$ and $B\in\gl2$,
\[
\big(W_1^{-1}F\big)(B) =|\det(C_B)|^{-1/2}
(F[\u0,M_B])(C_B).
\]
\end{prop}
\begin{proof}
For
any integrable $h$ on $\gl2$, the Haar integral
on $\gl2$ can be expressed as
\[
\int_{\gl2}h\,d\mu_{\gl2}=\int_{K_0}\int_{H_{(1,0)}}
h(MC)\,|\det(C)|\,d\mu_{H_{(1,0)}}(C)\,d\mu_{K_0}(M).
\]
Thus, for any $f\in L^2(\gl2)$,
\[
\begin{split}
\int_{K_0}\|W_1f[\u0,M]\|_{_{L^2(H_{(1,0)})}}^{\,\,\,2}
d\mu_{K_0}(M) & = \int_{K_0}\int_{H_{(1,0)}}
|(W_1f[\u0,M])(C)|^2d\mu_{H_{(1,0)}}d\mu_{K_0}(M)\\
& = \int_{K_0}\int_{H_{(1,0)}}
|f(MC)|^2|\det(C)|\,
d\mu_{H_{(1,0)}}\,d\mu_{K_0}(M)\\
& =\int_{\gl2}|f|^2\,d\mu_{\gl2}<\infty.
\end{split}
\] 
Hence, $W_1f\in L^2\big(K,L^2(H_{(1,0)})\big)$ and
$W_1$ is an isometry. It is clear $W_1$ is linear.
Therefore, the range of $W_1$ is a closed subspace
of $L^2\big(K,L^2(H_{(1,0)})\big)$. We need to show 
that the range of $W_1$ is all of 
$L^2\big(K,L^2(H_{(1,0)})\big)$. For any 
$k\in C_c(K_0)$ and $h\in C_c(H_{(1,0)})$,
define $F_{k,h}\in L^2\big(K,L^2(H_{(1,0)})\big)$
by $\big(F_{k,h}[\u0,M]\big)(C)=k(M)h(C)$. The
linear span of $\{F_{k,h}:k\in C_c(K_0),
h\in C_c(H_{(1,0)})\}$ is dense in 
$L^2\big(K,L^2(H_{(1,0)})\big)$. So we just need to show each $F_{k,h}$ is in $W_1L^2\big(\gl2\big)$.
For $k\in C_c(K_0)$ and $h\in C_c(H_{(1,0)})$,
let $f_{k,h}(B)=|\det(C_B)|^{-1/2}k(M_B)h(C_B)$,
for all $B\in\gl2$. Since $B\to (M_B,C_B)$ is a
homeomorphism of $\gl2$ with $K_0\times H_{(1,0)}$
and $B\to|\det(C_B)|^{-1/2}$ is continuous, 
$f_{k,h}\in C_c\big(\gl2\big)\subseteq
L^2\big(\gl2\big)$. Moreover, since $M_{MC}=M$ and
$C_{MC}=C$,
\[
\begin{split}
\big(W_1f_{k,h}[\u0,M]\big)(C) & =|\det(C)|^{1/2}
f_{k,h}(MC)=|\det(C)|^{1/2}|\det(C)|^{-1/2}k(M)h(C)\\
& = \big(F_{k,h}[\u0,M]\big)(C),
\end{split}
\]
for any $C\in H_{(1,0)}$ and $[\u0,M]\in K$. Thus,
$F_{k,h}\in W_1L^2\big(\gl2\big)$, for any 
$k\in C_c(K_0)$ and $h\in C_c(H_{(1,0)})$. This
implies $W_1$ is a unitary map onto 
$L^2\big(K,L^2(H_{(1,0)})\big)$.
Also, $W_1^{-1}F_{k,h}=f_{k,h}$ so, for any $B\in\gl2$,
\[
\begin{split}
W_1^{-1}F_{k,h}(B) & =f_{k,h}(B)=|\det(C_B)|^{-1/2}k(M_B)h(C_B)\\
&=|\det(C_B)|^{-1/2}
(F_{k,h}[\u0,M_B])(C_B).
\end{split}
\]
Since $\{F_{k,h}:k\in C_c(K_0),h\in C_c(H_{(1,0)}\}$
is total in $L^2\big(K,L^2(H_{(1,0)})\big)$, 
\[
\big(W_1^{-1}F\big)(B)=|\det(C_B)|^{-1/2}
(F[\u0,M_B])(C_B),
\]
for all $B\in\gl2$
 and 
$F\in L^2\big(K,L^2(H_{(1,0)})\big)$.
\end{proof}

Continuing with a fixed $g\in L^2(\R^*)$ satisfying
$\int_{\R^*}\big||\nu|^{1/2}g(\nu)\big|^2
d\mu_{\R^*}(\nu)=1$, let 
\[
\cH_g=W_1^{-1}L^2(K,\cK_g)\subseteq L^2\big(\gl2\big).
\]
Then $\cH_g$ is a closed subspace of $L^2\big(\gl2
\big)$ and $W_1:\cH_g\to L^2(K,\cK_g)$ is a unitary 
map. Note that we use the same notation for 
$W_1$ and its restriction to $\cH_g$. Recall 
from \eqref{pi_omega} the representation  $\pi^{(1,0)}$ of
$G_2$ acting on the Hilbert space $L^2\big(\gl2\big)$.
For $[\ux,A]\in G_2$ and $f\in L^2\big(\gl2\big)$,
\[
\pi^{(1,0)}[\ux,A]f(B)=e^{2\pi i(1,0) B^{-1}\ux}
f(A^{-1}B), \text{ for a.e. } B\in \gl2.
\]
\begin{prop}
Let $g\in L^2(\R^*)$ satisfy
$\int_{\R^*}\big||\nu|^{1/2}g(\nu)\big|^2
d\mu_{\R^*}(\nu)=1$. The subspace $\cH_g$ of
$L^2\big(\gl2\big)$ is $\pi^{(1,0)}$-invariant and 
the restriction of $\pi^{(1,0)}$ to $\cH_g$ is 
equivalent to $\tilde{\sigma}$ via the unitary map
$W_1:\cH_g\to L^2(K,\cK_g)$.
\end{prop}
\begin{proof}
Let $[\ux,A]\in G_2$. For $f\in\cH_g$, let $F=W_1f 
\in L^2(K,\cK_g)$. Since $L^2(K,\cK_g)$ is 
$\tilde{\sigma}$-invariant, 
$\tilde{\sigma}[\ux,A]F\in L^2(K,\cK_g)$ as well. Thus 
$W_1^{-1}\tilde{\sigma}[\ux,A]F=
W_1^{-1}\tilde{\sigma}[\ux,A]W_1f\in\cH_g$. 

For any $B\in\gl2$, let $\uomega=(1,0)B^{-1}$. By Proposition \ref{identities} (a) 
$M_B=\gamma(\uomega)$ and
$C_B=\gamma(\uomega)^{-1}B$. Then, using
\eqref{sigma_2} and 
$|\det\big(\gamma(\uomega)\big)|^{1/2}=
\|\uomega\|^{-1}$,
\[
\begin{split}
W_1^{-1}\tilde{\sigma}[\ux,A]F(B) & =|\det(C_B)|^{-1/2}
\big(\tilde{\sigma}[\ux,A]F[\u0,M_B]\big)(C_B)\\
& = 
\textstyle\frac{|\det(\gamma(\uomega))|^{1/2}}{|\det(B)|^{1/2}}
\big(\tilde{\sigma}[\ux,A]
F[\u0,\gamma(\uomega)]\big)(\gamma(\uomega)^{-1}B)\\
& =\textstyle
\frac{|\det(A)|^{1/2}}{|\det(B)|^{1/2}\|\uomega A\|}
e^{2\pi i\uomega\ux}\,
\lambda_{H_{(1,0)}}\big(\gamma(\uomega)^{-1}A\gamma(\uomega A)\big)
\big(F[\u0,\gamma(\uomega A)]\big)
(\gamma(\uomega)^{-1}B)\\
& = \textstyle
\frac{|\det(A)|^{1/2}}{|\det(B)|^{1/2}\|\uomega A\|}
e^{2\pi i\uomega\ux}\,
\big(F[\u0,\gamma(\uomega A)]\big)
(\gamma(\uomega A)^{-1}A^{-1}B)\\
& = \textstyle
\frac{|\det(A)|^{1/2}}{|\det(B)|^{1/2}\|\uomega A\|}
e^{2\pi i\uomega\ux}\,
\big(W_1f[\u0,\gamma(\uomega A)]\big)
\big(\gamma(\uomega A)^{-1}A^{-1}B\big).
\end{split}
\]
Before applying $W_1$, note that 
$|\det\big(\gamma(\uomega A)^{-1}A^{-1}B\big)|^{1/2}
=\frac{|\det(B)|^{1/2}\|\uomega A\|}{|\det(A)|^{1/2}}
$, which will cancel the first factor in the
previous expression. Therefore, recalling that
$\uomega=(1,0)B^{-1}$,
\[
W_1^{-1}\tilde{\sigma}[\ux,A]W_1f(B) = 
e^{2\pi i\uomega\ux}f(A^{-1}B)
=e^{2\pi i(1,0)B^{-1}\ux}f(A^{-1}B)=
\pi^{(1,0)}[\ux,A]f(B).
\]
This implies that $\cH_g$ 
is $\pi^{(1,0)}$-invariant and 
the restriction of $\pi^{(1,0)}$ to $\cH_g$ is 
equivalent to $\tilde{\sigma}$.
\end{proof}

Recall Theorem \ref{decomp_GLn_1}. The nontrivial
orbit in the one dimensional case is 
$\cO_1=\widehat{R}\setminus\{0\}$, which is
naturally identified with $\R^*$. Let
$\cD=L^2(\cO_1)\cap L^2(\R^*)$. Note that
\[
\cD=\left\{h:\widehat{\R}\setminus\{0\}\to\C\,\big|\,
h \text{ measurable },\int_{-\infty}^\infty
|h(t)|^2dt<\infty \text{ and }
\int_{-\infty}^\infty
|h(t)|^2\frac{dt}{|t|}<\infty\right\},
\] 
which can be considered as a dense subspace of
either $L^2(\cO_1)$ or $L^2(\R^*)$.
Fix $\{g_j:j\in J\}
\subseteq\cD$ such that $\{g_j:j\in J\}$ is an
orthonormal basis in $L^2(\cO_1)$. Let 
$w_2(\nu)=|\nu|^{1/2}$, for $\nu\in\widehat{\R}
\setminus\{0\}$. We state the following 
elementary lemma for later use.
\begin{lemma}\label{w_2_unitary_map}
The map $h\to w_2h$ is a unitary map of $L^2(\cO_1)$
onto $L^2(\R^*)$. In particular,
$\{w_2g_j:j\in J\}$ is an orthonormal basis of
$L^2(\R^*)$.
\end{lemma}

Identifying $H_{(1,0)}$ with $G_1$,
Theorem \ref{decomp_GLn_1}, with $n=1$, says that
$L^2(H_{(1,0)})=\sum_{j\in J}^{\oplus}\cK_{g_j}$.
Therefore, 
$L^2\big(K,L^2(H_{(1,0)})\big)=\textstyle
\sum_{j\in J}^{\oplus}L^2(K,\cK_{g_j})$. 
Applying $W_1^{-1}$, now considered as a 
unitary map of $L^2\big(K,L^2(H_{(1,0)})\big)$ 
onto $L^2\big(\gl2\big)$, we get a decomposition
of $L^2\big(\gl2\big)$.
\begin{prop}\label{decomposition_gl2}
Let $\{g_j:j\in J\}
\subseteq\cD$ be an
orthonormal basis in $L^2(\cO_1)$. Then each
$\cH_{g_j}$ is a closed $\pi^{(1,0)}$-invariant
subspace of $L^2\big(\gl2\big)$ and the
restriction of $\pi^{(1,0)}$ to $\cH_{g_j}$
is equivalent to $\sigma$. Moreover,
$L^2\big(\gl2\big)=\sum_{j\in J}^{\oplus}
\cH_{g_j}$.
\end{prop}

Recall from Proposition \ref{Intertwine}, if
$\eta\in L^2(\cO)$ satisfies $\|\eta\|_{_{L^2(\cO)}}
=1$, then there is an isometric linear map
$U_\eta:L^2\big(\gl2\big)\to L^2(G_2)$ that 
intertwines $\pi^{(1,0)}$ with $\lambda_{G_2}$.
The map $U_\eta$ is defined by
\[
U_\eta f[\uy,B]=\int_{\widehat{\R^2}}
\eta(\uomega)f\big(B\gamma(\uomega)\big)
e^{-2\pi i\uomega B^{-1}\uy}d\uomega,\,\,
\text{for all}\,\,[\uy,B]\in G_2, 
f\in L^2\big(\gl2\big).
\]
The steps we have taken to move from 
$L^2\big(K,L^2(\R^*)\big)$ to $L^2(G_2)$ are
summarized in the following diagram.
\[
\begin{tikzcd}
\left(L^2\big(K,L^2(\R^*)\big);\sigma\right)
\arrow[d,"V_g'"]\\
\left(L^2(K,\cK_g);\tilde{\sigma}\right)
\arrow[d,"W_1^{-1}"]\\
\left(L^2\big(\gl2\big);\pi^{(1,0)}\right)
\arrow[d,"U_\eta"]\\
\left(L^2(G_2);\lambda_{G_2}\right)
\end{tikzcd}
\]
The vertical maps are linear isometries from the
upper Hilbert space into the lower Hilbert space
intertwining the corresponding unitary
representations. Thus, $\Phi_{\eta,g}=U_\eta\circ W_1^{-1}
\circ V_g'$ is a linear isometry of 
$L^2\big(K,L^2(\R^*)\big)$ into $L^2(G_2)$ that
intertwines $\sigma$ with $\lambda_{G_2}$.

Let $F\in L^2\big(K,L^2(\R^*)\big)$. For 
$[\ux,A]\in G_2$,
\begin{equation}\label{Phi_1}
\begin{split}
\Phi_{\eta,g}F[\ux,A] & =
\int_{\widehat{\R^2}}
\eta(\uomega)\big(W_1^{-1}V_g'F\big)\big)\big(A\gamma(\uomega)\big)
e^{-2\pi i\uomega A^{-1}\ux}d\uomega\\
& = |\det(A)|\int_{\widehat{\R^2}}
\eta(\uomega A)\big(W_1^{-1}V_g'F\big)\big)\big(A\gamma(\uomega A)\big)
e^{-2\pi i\uomega\ux}d\uomega.
\end{split}
\end{equation}
Observe that $M_{A\gamma(\uomega A)}=
\gamma(\uomega)$ and 
$C_{A\gamma(\uomega A)}=
\gamma(\uomega)^{-1}A\gamma(\uomega A)$. So
\[
\det(C_{A\gamma(\uomega A)})=
\det\big(\gamma(\uomega)^{-1}A\gamma(\uomega A)\big)
=\frac{\|\uomega\|^2\det(A)}{\|\uomega A\|^2}
\]
Thus,
\[
\begin{split}
\big(W_1^{-1}V_g'F\big)\big)&\big(A\gamma(\uomega A)\big)  =
\textstyle\frac{\|\uomega A\|}{\|\uomega\|
\cdot|\det(A)|^{1/2}}V_g(F[\u0,\gamma(\uomega)])
\big(\gamma(\uomega)^{-1}A\gamma(\uomega A)
\big)\\
& =\textstyle\frac{\|\uomega A\|}{\|\uomega\|
\cdot|\det(A)|^{1/2}}\langle F[\u0,\gamma(\uomega)],
\pi^1\big(\gamma(\uomega)^{-1}A\gamma(\uomega A)
\big)g\rangle_{_{L^2(\R^*)}}\\
& = {\textstyle\frac{\|\uomega A\|}{\|\uomega\|
\cdot|\det(A)|^{1/2}}}\int_{\R^*}
(F[\u0,\gamma(\uomega)])(\nu)\overline{  
\pi^1\big(\gamma(\uomega)^{-1}A\gamma(\uomega A)
\big)g(\nu)}\,d\mu_{\R^*}(\nu).
\end{split}
\] 
Inserting this into \eqref{Phi_1} gives\\
\[
\begin{split}
\Phi_{\eta,g} & F[\ux,A] =\\
& \int_{\widehat{\R^2}}\int_{\R^*}
{\textstyle\frac{|\det(A)|^{1/2}\|\uomega A\|}{\|\uomega\|}}
\eta(\uomega A)
e^{-2\pi i\uomega\ux}
(F[\u0,\gamma(\uomega)])(\nu)\overline{ 
\pi^1\big(\gamma(\uomega)^{-1}A\gamma(\uomega A)
\big)g(\nu)}\,d\mu_{\R^*}(\nu)\,d\uomega.
\end{split}
\]

We will compare the expression for $\Phi_{\eta,g}F$
with a coefficient function of the irreducible
representation $\sigma$. 

If  
$E\in L^2\big(K,L^2(\R^*)\big)$ is fixed, then,
for any $F\in L^2\big(K,L^2(\R^*)\big)$, $V_EF$
is the continuous function on $G_2$ defined
by $V_EF[\ux,A]=
\langle F,\sigma[\ux,A]E\rangle_{_{L^2(K,L^2(\R^*))}}$, 
for all $[\ux,A]\in G_2$. Recall that the Haar integral over $K$ can be expressed using the parametrization
$\uomega\to\gamma(\uomega)$ by $\cO$, which is
co-null in $\widehat{\R^2}$. Then, for $[\ux,A]
\in G_2$,
\begin{equation}\label{V_EF}
\begin{split}
V_EF[\ux,A] & =
\langle F,\sigma[\ux,A]E\rangle_{_{L^2(K,L^2(\R^*))}}
\\
& =\int_{\widehat{\R^2}}\int_{\R^*} 
(F[\u0,\gamma(\uomega)])(\nu)\overline{\sigma[\ux,A]
(E[\u0,\gamma(\uomega)])(\nu)}\,d\mu_{\R^*}(\nu)
\frac{d\uomega}{\|\uomega\|^2}.
\end{split}
\end{equation}
Note that 
\[
\overline{\sigma[\ux,A]
(E[\u0,\gamma(\uomega)])(\nu)}=
\frac{|\det(A)|^{1/2}\|\uomega\|}{\|\uomega A\|}
e^{-2\pi i\uomega\ux}\,
\overline{\pi^1\big(\gamma(\uomega)^{-1}
A\gamma(\uomega A)\big)E[\u0,\gamma(\uomega A)](\nu)
}.
\]
 If we select $E$ as a function built from
$\eta$ and $g$, then we can make the expression for
$V_EF$ coincide with that for $\Phi_{\eta,g}$. Define
$E$ as follows: For $[\u0,L]\in K$ and $\nu\in\R^*$,
\[
E[\u0,L](\nu)=
\textstyle\frac{\overline{\eta}((1,0)L^{-1})}{|\det(L)|}\,g(\nu).
\]
If $L=\gamma(\uomega A)$, then 
$E[\u0,\gamma(\uomega A)](\nu)=
\frac{\overline{\eta}(\uomega A)}{|\det(\gamma(\uomega A))|}\,g(\nu)=\|\uomega A\|^2
\overline{\eta}(\uomega A)\,g(\nu)$. Thus  
\[
\overline{\sigma[\ux,A] 
(E[\u0,\gamma(\uomega)])(\nu)}\frac{1}{\|\uomega\|^2}=
\frac{|\det(A)|^{1/2}\|\uomega A\|}{\|\uomega\|}
\eta(\uomega A)\,e^{-2\pi i\uomega\ux}\,
\overline{\pi^1\big(\gamma(\uomega)^{-1}
A\gamma(\uomega A)\big)g(\nu).
}
\]
Substituting into \eqref{V_EF}, there is a perfect
match with the expression for $\Phi_{\eta,g}F$. Thus, we have the following theorem.
\begin{theorem}\label{V_E_isometry}
Let $g\in L^2(\R^*)$ satisfy $\int_{\R^*}\big|
|\nu|^{1/2}g(\nu)\big|^2d\mu_{\R^*}(\nu)=1$ and let
$\eta\in L^2(\widehat{\R^2})$ satisfy
$\|\eta\|_{_{L^2(\widehat{\R^2})}}=1$.
Let $E\in  L^2\big(K,L^2(\R^*)\big)$ be defined
as 
\[
E[\u0,L](\nu)=
\frac{\overline{\eta}((1,0)L^{-1})}{|\det(L)|}\,g(\nu),
\text{ for each } [\u0,L]\in K \text{ and }
\nu\in\R^*.
\]
 Define $V_EF[\ux,A]=
\langle F,
\sigma[\ux,A]E\rangle_{_{L^2(K,L^2(\R^*))}}$, for
$[\ux,A]\in G_2$ and 
$F \in L^2\big(K,L^2(\R^*)\big)$. Then $V_E$ is
a linear isometry of $L^2\big(K,L^2(\R^*)\big)$
into $L^2(G_2)$ that intertwines $\sigma$ with
$\lambda_{G_2}$. In particular, $\sigma$ is a square-integrable representation of $G_2$.
\end{theorem}

Combining Theorem \ref{decomp_GLn_1} with Proposition
\ref{decomposition_gl2} will now provide a
decomposition of the regular representation,
$\lambda_{G_2}$ on $L^2(G_2)$, into infinitely many
copies of $\sigma$. Fix an orthonormal
basis $\{g_j:j\in J\}$
of $L^2(\cO_1)$ consisting of functions
in $\cD$ as in Proposition \ref{decomposition_gl2}.
For each $j\in J$, $\cK_{g_j}=V_{g_j}L^2(\R^*)$
and $\cH_{g_j}=W_1^{-1}L^2(K,\cK_{g_j})$. Then
$\cH_{g_j}$ is a closed $\pi^{(1,0)}$-invariant
subspace of $L^2\big(\gl2\big)$ and the
restriction of $\pi^{(1,0)}$ to $\cH_{g_j}$ is
equivalent to $\sigma$ via $W_1^{-1}\circ V_g'$.
Moreover, Proposition \ref{decomposition_gl2} says
$L^2\big(\gl2\big)=\sum_{j\in J}^{\oplus}
\cH_{g_j}$.

On the other hand, $\cO=\widehat{\R^2}\setminus 
\{(0,0)\}$ is the orbit of $(1,0)$ in 
$\widehat{\R^2}$ under the action of $\gl2$. It is
equipped with the Lebesgue measure of 
$\widehat{\R^2}$. Fix an orthonormal basis 
$\{\eta_i:i\in I\}$ of
$L^2(\widehat{\R^2})=L^2(\cO)$. At this point, we pause to formulate an analog of Lemma \ref{w_2_unitary_map}. There is a useful conjugate linear
map $W_\gamma$ from $L^2(\cO)$ to $L^2(K)$ provided
by the map $\gamma$. We observe that $\gamma^{-1}:
K_0\to\cO$ is such that $\gamma^{-1}(L)=(1,0)L^{-1}$,
for all $L\in K_0$. For $\xi\in L^2(\cO)$, let
$W_\gamma\xi[\u0,L]=
|\det(L)|^{-1/2}\overline{\xi}\big(\gamma^{-1}(L)\big)$, 
for a.e. $L\in K$.
\begin{lemma}\label{W_gamma_unitary}
The map $\xi\to W_\gamma\xi$ is a conjugate linear
isometry of 
$L^2(\cO)$ onto $L^2(K)$. In particular, 
$\{W_\gamma\eta_i:i\in I\}$ is an orthonormal basis of
$L^2(K)$.
\end{lemma}
\begin{proof}
We use the expression for Haar integration on $K$ given
in \eqref{Haar_K}. Then, for $\xi\in L^2(\cO)$,
\[
\begin{split}
\int_K|W_\gamma\xi|^2d\mu_K & =\int_{\widehat{\R^2}}
\left|W_\gamma\xi[\u0,\gamma(\uomega)]\right|^2
\frac{d\uomega}{\|\uomega\|^2}=\int_{\widehat{\R^2}}
\left||\det\big(\gamma(\uomega)\big)|^{-1/2}
\overline{\xi}\big((1,0)\gamma(\uomega)^{-1}\big)\right|^2
\frac{d\uomega}{\|\uomega\|^2}\\
&= \int_{\widehat{\R^2}}
\|\uomega\|^2\left|
\overline{\xi}\big((1,0)\gamma(\uomega)^{-1}\big)\right|^2
\frac{d\uomega}{\|\uomega\|^2}=
\int_{\widehat{\R^2}}|\xi(\uomega)|^2d\uomega=
\|\xi\|_{L^2(\cO)}^{\,\,2}.
\end{split}
\]
Thus, $W_\gamma\xi\in L^2(K)$ and 
$\|W_\gamma\xi\|_{L^2(K)}^{\,\,2}=
\|\xi\|_{L^2(\cO)}^{\,\,2}$. Clearly, $W_\gamma$ is
additive and $W_\gamma\alpha\xi=
\overline{\alpha}W_\gamma\xi$ for all $\alpha\in\C$ and
$\xi\in L^2(\cO)$. Moreover, a calculation similar to 
that above shows that $W_\gamma$ maps $L^2(\cO)$ 
onto $L^2(K)$ and
$W_\gamma^{-1}$ is given by
$W_\gamma^{-1}f(\uomega)=
\|\uomega\|^{-1}\overline{f}[\u0,\gamma(\uomega)]$, for
a.e. $\uomega\in\cO$ and $f\in L^2(K)$.
\end{proof}

For each $i\in I$,
$U_{\eta_i}:L^2(\gl2)\to L^2(G_2)$ is given by
\[
U_{\eta_i}f[\uy,B]=\int_{\widehat{\R^2}}
\eta_i(\uomega)f\big(B\gamma(\uomega)\big)
e^{-2\pi i\uomega B^{-1}\uy}\,d\uomega,
\]
for $[\uy,B]\in G_2$ and $f\in L^2\big(\gl2\big)$.
By Theorem \ref{decomp_GLn_1}, each $U_{\eta_i}
L^2\big(\gl2\big)$ is a $\lambda_{G_2}$-invariant
closed subspace of $L^2(G_2)$ and $U_{\eta_i}$
intertwines $\pi^{(1,0)}$ with the restriction
of $\lambda_{G_2}$ to $U_{\eta_i}
L^2\big(\gl2\big)$ and $L^2(G_2)=\sum_{i\in I}^{\oplus}
U_{\eta_i}L^2\big(\gl2\big)$.

For each $(i,j)\in I\times J$, form $E_{i,j}
\in L^2\big(K,L^2(\R^*)\big)$ by
\begin{equation}\label{E_ij}
E_{i,j}[\u0,L](\nu)= 
\textstyle\frac{\overline{\eta_i}((1,0)L^{-1})}{|\det(L)|}\,g_j(\nu),
\end{equation}
for $\nu\in\R^*$ and $[\u0,L]\in K$. The orthogonal 
decompositions just recalled imply the following
theorem. 
\begin{theorem}\label{V_E_ij}  
Let $\sigma$ be the representation
of $G_2$ on the Hilbert space 
$L^2\big(K,L^2(\R^*)\big)$ 
given in \eqref{sigma_pointwise_formula}.
Let $\{\eta_i:i\in I\}$ be an orthonormal basis of
$L^2(\cO)$ and let $\{g_j:j\in J\}$ be
an orthonormal basis of $L^2(\cO_1)$ consisting
of functions in $L^2(\R^*)$. For each $(i,j)
\in I\times J$, form $E_{i,j}
\in L^2\big(K,L^2(\R^*)\big)$ as in \eqref{E_ij}
and let $\cM_{i,j}=
V_{E_{i,j}}L^2\big(K,L^2(\R^*)\big)$. Then
each $\cM_{i,j}$ is a closed $\lambda_{G_2}$-invariant
subspace of $L^2(G_2)$ and $V_{E_{i,j}}$ is an
isometry that intertwines $\sigma$ with the
restriction of $\lambda_{G_2}$ to $\cM_{i,j}$.
Moreover
$L^2(G_2)=\sum_{(i,j)\in I\times J}^{\oplus}\cM_{i,j}.$
\end{theorem}

\section{The Plancherel Formula for $G_2$} 
For certain classes of locally compact groups, 
an abstract Plancherel Theorem has been established.
For example, Theorem 18.8.2 of \cite{Dix} implies
that,
if $G$ is a Type I, separable,
unimodular group, then there exists a unique measure 
$\mu_{\widehat{G}}$ on $\widehat{G}$ such that,
for any $f\in L^1(G)\cap L^2(G)$, $\pi(f)$
is a Hilbert–-Schmidt operator on $\cH_\pi$, for 
$\mu_{\widehat{G}}$-a.e. $\pi\in\widehat{G}$, and
\begin{equation}\label{Plancherel_unimodular}
\|f\|_{L^2(G)}^{\,\,\,2}=
\int_{\widehat{G}}\|\pi(f)\|_{\rm HS}^{\,\,2}
\,d\mu_{\widehat{G}}(\pi).
\end{equation}
 There are generalizations of
\eqref{Plancherel_unimodular} to appropriate
classes of nonunimodular groups in \cite{KL} and
\cite{DM} (see \cite{Fuh} for an organized treatment
and some extensions). The group $G_2$ satisfies
the hypotheses for the results in both \cite{KL} and
\cite{DM}. Using Theorem \ref{V_E_ij}, the
Plancherel formula for $G_2$ is quite concrete.
The method below is a special case of Theorem 2.34
in \cite{Fuh}.

We fix $\{\eta_i:i\in I\}$, an orthonormal basis of
$L^2(\cO)$, and $\{g_j:j\in J\}$,
an orthonormal basis of $L^2(\cO_1)$ consisting
of functions in $L^2(\R^*)$ as in Theorem \ref{V_E_ij}.
By Lemma \ref{w_2_unitary_map}, $\{w_2g_j:j\in J\}$
is an orthonormal basis of $L^2(\R^*)$. Likewise, by 
Lemma \ref{W_gamma_unitary}, $\{W_\gamma\eta_i:i \in I\}$
is an orthonormal basis of $L^2(K)$. Since we can
view $L^2\big(K,L^2(\R^*)\big)$ as $L^2(K\times\R^*)$,
where $K\times\R^*$ is given the product of the
Haar measures, and $L^2(K\times\R^*)$ can be identified
with $L^2(K)\otimes L^2(\R^*)$ (see Example 2.6.11 of
\cite{KR}), we can construct an orthonormal basis of
$L^2\big(K,L^2(\R^*)\big)$ from the orthonormal
basis $\{(W_\gamma\eta_i)\otimes(w_2g_j):(i,j)\in 
I\times J\}$ of $L^2(K)\otimes L^2(\R^*)$.
For each 
$(i,j)\in I\times J$, define $F_{i,j}\in 
L^2\big(K,L^2(\R^*)\big)$ by
$\big(F_{i,j}[\u0,L]\big)(\nu)=
\big(W_\gamma\eta_i[\u0,L]\big)\big((w_2g_j)(\nu)\big)$,
for a.e. $\nu\in \R^*$ and $[\u0,L]\in K$. That is
$\big(F_{i,j}[\u0,L]\big)(\nu)=
\frac{\overline{\eta_i}\big((1,0)L^{-1}\big)}{|\det(L)|^{-1/2}}|\nu|^{1/2}g_j(\nu)$, for a.e. $[\u0,L]$ and 
$\nu$. Then $\{F_{i,j}:(i,j)\in I\times J\}$ is an
orthonormal basis of $L^2\big(K,L^2(\R^*)\big)$. Note
the close relationship with the $E_{i,j}$ as defined
in \eqref{E_ij}. Define the positive unbounded operator
$T$ on $L^2\big(K,L^2(\R^*)\big)$ by, for 
$F\in L^2\big(K,L^2(\R^*)\big)$, 
$\big((TF)[\u0,L]\big)(\nu)=|\det(L)\nu|^{-1/2}
\big(F[\u0,L]\big)(\nu)$, for
a.e. $[\u0,L]$ and $\nu$. Then $F_{i,j}$ is in the
domain of $T$ and
$E_{i,j}=TF_{i,j}$, for all $(i,j)\in I\times J$.

Since $\{F_{i,j}:(i,j)\in I\times J\}$ is an
orthonormal basis of $L^2\big(K,L^2(\R^*)\big)$,
Theorem \ref{V_E_ij} implies that
$\left\{
V_{E_{i,j}}F_{i',j'}:(i,j),(i',j')\in I\times J
\right\}$
is an orthonormal basis of $L^2(G_2)$. Therefore,
for any $f\in L^2(G_2)$,
\begin{equation}\label{double_sum}
\|f\|_{L^2(G_2)}^{\,\,2}=\sum_{(i,j)\in I\times J}\,\,
\sum_{(i',j')\in I\times J}\left|\langle f,
V_{E_{i,j}}F_{i',j'}\rangle_{L^2(G_2)}\right|^2.
\end{equation}
But, if $f\in L^1(G_2)\cap L^2(G_2)$, then we can
further analyze each of the inner products in the sum in 
\eqref{double_sum}. For $(i,j),(i',j')\in I\times J$,
\begin{equation*}
\begin{split}
\langle f,
V_{E_{i,j}}F_{i',j'}\rangle_{L^2(G_2)} & =
\int_{G_2}f[\ux,A]\overline{\langle F_{i',j'},
\sigma[\ux,A]E_{i,j}\rangle_{L^2(K,L^2(\R^*))}}
\,d\mu_{G_2}[\ux,A]\\
& = \int_{G_2}f[\ux,A]\langle
\sigma[\ux,A]E_{i,j},F_{i',j'}\rangle_{L^2(K,L^2(\R^*))}
\,d\mu_{G_2}[\ux,A]\\
& = \langle
\sigma(f)E_{i,j},F_{i',j'}\rangle_{L^2(K,L^2(\R^*))}=
\langle
\sigma(f)TF_{i,j},F_{i',j'}\rangle_{L^2(K,L^2(\R^*))}.
\end{split}
\end{equation*}
So, for $(i,j)\in I\times J$ fixed,
\begin{equation*}
\sum_{(i',j')\in I\times J}\left|\langle f,
V_{E_{i,j}}F_{i',j'}\rangle_{L^2(G_2)}\right|^2=
\sum_{(i',j')\in I\times J}\left|
\sigma(f)TF_{i,j},F_{i',j'}\rangle_{L^2(K,L^2(\R^*))}
\right|^2.
\end{equation*}
Since $\{F_{i',j'}:(i',j')\in I\times J\}$ is an
orthonormal basis of $L^2\big(K,L^2(\R^*)\big)$,
\begin{equation}\label{single_sum}
\sum_{(i',j')\in I\times J}\left|\langle f,
V_{E_{i,j}}F_{i',j'}\rangle_{L^2(G_2)}\right|^2=
\|\sigma(f)TF_{i,j}\|_{L^2(K,L^2(\R^*))}^{\,\,2}.
\end{equation}
Inserting \eqref{single_sum} into \eqref{double_sum}
gives $\|f\|_{L^2(G_2)}^{\,\,2}=
\sum_{(i,j)\in I\times J}
\|\sigma(f)TF_{i,j}\|_{L^2(K,L^2(\R^*))}^{\,\,2}=
\|\sigma(f)T\|_{\rm HS}^{\,2}$, when 
$f\in L^1(G_2)\cap L^2(G_2)$. This establishes a
concrete version of the abstract Plancherel Theorem
for the non-unimodular group $G_2$ and serves as an
example of the considerably more general Theorem 2.34
of \cite{Fuh}.
\begin{prop}\label{Plancherel_G_2}
Define the positive operator
$T$ on $L^2\big(K,L^2(\R^*)\big)$ by, 
\[
\big((TF)[\u0,L]\big)(\nu)=|\det(L)\nu|^{-1/2}
\big(F[\u0,L]\big)(\nu) \text{ for. a.e. }
([\u0,L],\nu)\in K\times\R^*,
\] 
for $F\in L^2\big(K,L^2(\R^*)\big)$. Let $\sigma$ be
the irreducible representation of $G_2$ given by \eqref{sigma_pointwise_formula}. If 
$f\in L^1(G_2)\cap L^2(G_2)$, then $\sigma(f)T$ extends
to a Hilbert-Schmidt operator on 
$L^2\big(K,L^2(\R^*)\big)$ and 
$\|f\|_{L^2(G_2)}=
\|\sigma(f)T\|_{\rm HS}$.
\end{prop}

\section{Concluding Remarks}
The bases $\{\eta_i:i\in I\}$ of $L^2(\cO)$ and
$\{g_j:j\in J\}$ of $L^2(\cO_1)$ can be selected
from functions in $C_c(\cO)$ and $C_c(\cO_1)$,
respectively, and as smooth as one may need. The
subspaces $\{\cM_{i,j}:(i,j)\in I\times J\}$ are
then constructed from these bases. The decomposition in
Theorem \ref{V_E_ij} should be useful for analysis
on $G_2$ in a manner similar to how the Peter--Weyl
Theorem plays a role in analysis on a compact group.

The representation $\sigma$ can be presented in 
several equivalent forms using natural isomorphisms
of $L^2\big(K,L^2(\R^*)\big)$ with related function 
spaces. For example, $L^2\big(K,L^2(\R^*)\big)$ is
identified with $L^2\big(K_0,L^2(\R^*)\big)$ with a
simple change of notation. Also,
$L^2\big(K,L^2(\R^*)\big)$ can
be viewed as $L^2(K_0\times\R^*)$ by mapping
$F\to \tilde{F}$, where $\tilde{F}(L,\nu)=
\big(F[\u0,L]\big)(\nu)$, for a.e. $(L,\nu)\in K_0\times
\R^*$. The expressions 
for $\sigma$ can be easily modified.
A more substantial change of Hilbert space is
carried out in \cite{MT}, a companion paper to this one.
There, we define a unitary map $U$ of 
$L^2\big(K,L^2(\R^*)\big)$ onto 
$L^2\big(\widehat{\R^2}\times\widehat{\R}\big)$
and let $\sigma_2[\ux,A]=U\sigma[\ux,A]U^{-1}$, for
all $[\ux,A]\in G_2$. We apply the conclusions of
the Duflo--Moore Theorem to $\sigma_2$ and obtain
an analog of the continuous wavelet transform for
functions of three variables, with one of the variables
treated in a manner distinct from the other two.

\end{document}